\documentclass[12pt,twoside]{article}
\usepackage[dvips]{graphics}
\usepackage{amsmath,amsfonts,amssymb,amsthm}
\usepackage{epsfig,fullpage}
\usepackage{xspace}
\usepackage{color}

\newcommand{\dt}{\Delta t}

\newtheorem{thrm}{Theorem}

\newtheorem{algorithm}{Algorithm}
\newcommand{\uh}{U}

\newtheorem{theorem}{Proposition}

\def\N{\mathbb N}
\def\Z{\mathbb Z}
\def\R{\mathbb R}

\begin{document}
\title{Fast and Efficient Numerical Methods for an Extended  Black-Scholes Model}
\author{Samir Kumar Bhowmik~\footnote{The author would like to thank Chris C. Stolk, the KdV institute for Mathematics, University of Amsterdam for introducing him wavelet and Fourier sine preconditioners for elliptic operators.}\\
Department of Mathematics, University of Dhaka, Dhaka 1000, Bangladesh\\
Bhowmiksk@gmail.com}

\maketitle

\begin{abstract}
 An efficient linear solver plays an important role while solving partial differential equations (PDEs) and partial integro-differential equations (PIDEs) type mathematical models. In most cases, the efficiency depends on the stability and accuracy of the numerical scheme considered. In this article we consider a PIDE that arises in option pricing theory (financial  problems) as well as in various scientific modeling
 and deal with two different topics. In the first part of the article, we study   several iterative techniques (preconditioned) for the PIDE model. A wavelet basis  and a Fourier sine basis have been used to design various  preconditioners to improve the convergence criteria of iterative solvers. We implement a  multigrid (MG) iterative method.  In fact, we approximate the problem using a finite difference scheme, then implement a few  preconditioned Krylov subspace methods as well as a MG method to speed up the computation. Then, in the second part in this study, we  analyze  the stability and the accuracy of two different one step schemes to approximate the model.

\end{abstract}


\section{Introduction}
The pricing of options is a central problem in financial investment. It is important in both theoretical and practical point of view since the use of options thrives in the financial market. In option pricing theory, the study of the Black-Scholes equation is very important and interesting (study of  a parabolic partial differential equation (PDE)). In recent days, researchers have extended the model by looking at the nonlocal effects, which is a linear partial integro-differential equation (PIDE).

We consider such a partial integro-differential equation~\cite{rcontevoltchkova2005, D.J.Duffy03}
\begin{equation}\label{f:generalmodel01}
 \frac{\partial u(x, t)}{\partial t} = \mathcal{L}u(x, t),
 \end{equation}
 where
 \[
  \mathcal{L} u(x, t) = \sigma  \frac{\partial^2 u(x, t)}{\partial x^2}+ \mu\frac{\partial u(x, t)}{\partial x}- r u(x, t)
 + \lambda  \int_{\Omega}\ J(x-y)\left( u(y,t)-u(x,t)
   \right)dy,
\]
with initial condition
\[
u(x, 0) =\psi(x), \ -\infty< x <\infty.
\]
Here $\sigma \ge 0$, $\lambda \ge 0$, with  $(\sigma, \lambda) \ne (0, 0)$, $ r \ge 0$ and $\mu \in \R$, $J$ is the kernel of the model and $u=u(x, t)$ represents the option price (contingent claim).
A normalized kernel function $J(x)$, i.e., $\int_{\Omega} J(x)dx=1$ has been considered in most of the
models~\cite{Dug, F.Fiorani, J.Medlock} with suitable parameter values.
In general, $J(x-y)$ is a kernel function that models  the interaction between options at positions $x$ and  $y$. The effect of close neighbours
$x$ and $y$ is usually greater than that from more distant ones;
this is incorporated in $J$.  For simplicity we assume that $J$ is a non-negative function that satisfies smoothness, symmetry and decay conditions. One may consider any $J$ to implement the schemes we discuss in this study.
$\frac{1}{2}e^{-|x|}$ and $\sqrt{\frac{\omega}\pi}e^{-\omega x^2}$ are two sample kernel functions.
Boundary conditions are always an issue in these types of models. Here one may easily consider BCs~\cite{D.J.Duffy03}
\[
\frac{\partial^2 u}{\partial x^2} = 0, \ \text{as } \ x\rightarrow \pm \infty.
\]

Operator defined by (\ref{f:generalmodel01}) with $\sigma=0, \mu=0, \ r=0$
comes while modeling phase transitions~\cite{Dug}, dynamics of neurons in the
brain model~\cite{Gab01, C.Chow}, and population dynamics models~\cite{J.Medlock} as well.

Numerical approximation and analysis  of PDEs and PIDEs using finite difference, finite element method and the pseudo-spectral method are of  ongoing research interest. Specially for PIDEs, fast and efficient numerical tools are still to be developed.
A clear introduction about option pricing models and some finite difference schemes to approximate the models can be found in \cite{D.J.Duffy03, F.Fiorani}. 

A noble study about the model problem  \eqref{f:generalmodel01} can be found in \cite{F.Fiorani}. The authors consider a European and an American vanilla and barrier options based on the variance gamma process. They discuss derivation of  \eqref{f:generalmodel01} in detail and approximate the model problem numerically  by implementing a finite difference algorithm. They present some numerical experiments on the option pricing. But no efficient linear algebra solvers for the discrete equivalent of the model as well as the  stability and the accuracy analysis of the approximation are discussed.

In \cite{Dug}, Dugald et. al. consider a nonlocal model of phase transitions of type (\ref{f:generalmodel01}) ($\sigma=0, \ \mu=0, \ r=0$). Stability of stationary solution and  coarsening of solutions have been discussed by the authors. They present a finite element scheme to solve the problem and discuss some experimental results.

In \cite{SKB02}, the author also considers the nonlocal model of phase
transitions. He approximates the problem using the forward Euler scheme and examines the convergence rate
of the scheme.

A convolutional model of $\dot \theta$ Neuron network has been considered in  \cite{SamirKumarBhowmik04}. The author approximates the problem using finite element method in space, then he applies  implicit schemes for time stepping.  Then the author analyzes the error in such an approximation.

The PIDE model (\ref{f:generalmodel01}) is well studied in \cite{rcontevoltchkova2005}.
They discuss  viscosity   solution of the model followed by  a few finite difference approximations. They show that the infinite domain can be truncated to a finite domain $[A, B]$ where $A$ and $B$ depend  on the decay of the kernel function $J(x)$. Thus the problem can be considered as a IBVPs.
Considering the kernel of the convolution integral as
$$
J(x)\equiv J_{\delta}(x) = \sqrt{\frac{1}{2\pi \delta}} \exp \left(-\frac{y^2}{2\delta^2}\right),
$$
the authors in \cite{rcontevoltchkova2005}   formulate
$$
 A = +\sqrt{-2\delta^2 \log (\delta \varepsilon \sqrt{2\pi})},
$$
and
$
B= -A,
$
where $\varepsilon>0$ is considered so that $J_{\delta}(x) \ge \varepsilon$.
One may consider the model in a spatial periodic domain~\cite{SKB02} as well.
We use these concepts to approximate the model in a finite as well as a periodic spatial domain.

There are many other articles those discuss these type of models, but to the best of our  knowledge the  discussion about efficient linear solvers for this type of models is absent. So we focus on some fast and efficient numerical schemes as well as the stability and the accuracy analysis of two  finite difference schemes for the operator acting on (\ref{f:generalmodel01}).
We start the study by approximating the problem using  the backward Euler in time  for (\ref{f:generalmodel01}) and investigate some linear algebra tools to speed up the computational process in $\Omega\subset \mathbb{R}$. Then we analyze the stability and the accuracy of two different schemes considering $\Omega=\mathbb{R}$.
The article is organized in the following way.  We propose and implement several efficient linear system solvers to compute solutions of (\ref{f:generalmodel01}) in Section~\ref{section01}.
Then we discuss the stability of an explicit and an semi-implicit scheme in Section~\ref{section02}.  We use Fourier transforms of the integro-differential equation  for our analysis throughout this study. The accuracy analysis of two full discrete schemes as well as a  semi-discrete approximation are presented in Section~\ref{f:section04} and Section~\ref{f:semidiscretesection}, respectively.
We finish the article with discussion, conclusions and open problems in Section~\ref{section03}.
%
\section{Numerical approximation}\label{section01}
Several standard ordinary differential equation solvers are available and can be used to approximate the time derivative. So  our main goal, in this study, is to approximate the model \eqref{f:generalmodel01} in space domain.  Here we first perform a time integration, then look for some fast and efficient space integration tools.

Now one may start with the forward Euler scheme for time stepping (an explicit scheme), which uses the values of only previous time step to calculate those of the next. The Algorithm is very simple, in that each unknown, at time step $n+1$, is calculated independently, so it  does not require simultaneous solution of equations, and can even be performed easily. But it is unstable for large time steps. We have analyzed the stability condition, and the accuracy of such a scheme in Section~\ref{section02} and thereafter.
Instabilities are big problems in numerical approximation. We want to use large time steps and so we are interested in using implicit schemes.

\subsection{An  implicit scheme}
We start with the  implicit Euler scheme for time integration.  We approximate the model \eqref{f:generalmodel01} in time by %
\[
 - \Delta t\sigma  \frac{\partial^2 u^n(x)}{\partial x^2} - \Delta t \mu\frac{\partial u^n(x)}{\partial x}+(1 + r\Delta t)u^{n}(x) - \Delta t \lambda  \int_{\Omega}\ J(x-y)\left( u^n(y)-u^n(x)
   \right)dy=u^{n-1}(x),
\]
where $u^{n}(x)=u(x, t_n)$, $n \ge 0$. We will demonstrate several schemes to approximate the semi-discrete spatial model.
For simplicity we write
\begin{equation}\label{f:generalmodel_dt_01}
\mathcal{L}u^n(x)\equiv \mathcal{L}_1(u^n(x)) + \mathcal{L}_2(u^n(x)) =u^{n-1}(x),
\end{equation}
where,
\[
 \mathcal{L}_1(u^n(x)) = -\Delta t\sigma  \frac{\partial^2 u^n(x)}{\partial x^2} -\Delta t \mu\frac{\partial u^n(x)}{\partial x}+(1+r\Delta t)u^{n}(x),
\]
and
\[
\mathcal{L}_2(u^n(x)) =  -\Delta t \lambda  \int_{\Omega}\ J(x-y)\left( u^n(y)-u^n(x)
   \right)dy.
\]
It is easy to verify that the operator $\mathcal{L}$ acting on \eqref{f:generalmodel_dt_01} is an elliptic partial differential operator~\cite{LCEvans1998}.

Now it is our aim to design and implement some fast and efficient solvers for \eqref{f:generalmodel_dt_01}.
We start by approximating
\[
 \frac{\partial^2 u^n(x)}{\partial x^2} = \frac{U^n_{i+1}-2 U^n_i + U^n_{i-1}}{h^2}, \quad
 \frac{\partial u^n(x)}{\partial x} = \frac{U^n_{i+1}- U^n_{i-1} }{2 h},
\]
 and  
\begin{align*}
 \mathcal{L}_1 u^n(x_i) &= \sum_{j=-\infty}^{\infty}\int_{\Omega_i} J(x_i-y)(U^n(x_i)-U^n(y))dy\\
  &\approx
 \sum_{j=-N/2}^{N/2-1}h J(x_i-x_j)(U^n(x_i)-U^n(x_j)).
\end{align*}
Based on these approximations we write the full discrete model as
\begin{equation}\label{f:linearsystem_01}
 AU^n= U^{n-1},
\end{equation}
 which is a system of linear equations with unknown $U^n$.
 The symbol of the discrete equivalent of $\mathcal{L}$ can be written as (see Section~\ref{section02})
 \[
 A_{syb}(\dt, h\xi) =  \dt \left(1-\tilde q(\xi) +\frac{4}{h^2}\sin^2 \left(\frac{h\xi}{2}\right)-\frac{i}{h}\sin (h\xi) \right).
 \]
 Considering $$g(\dt, h\xi)=\frac{1}{ A_{syb}(\dt, h\xi)}$$ the unknown can be expressed in the Fourier domain as (see Section~\ref{section02} for details)
 \[
 \tilde U^n(\xi)= g^n(\dt, h\xi) \tilde U^0(\xi).
 \]
 Since, $$|A_{syb}(\dt, h\xi)|\ge 1$$ for any choice of $\dt$ and $h$, the scheme is unconditionally stable (a few discrete symbols of this type of operators have been evaluated in detail in next section).

Now the main difficulty of solving linear systems like \eqref{f:linearsystem_01}  is that the maximal eigenvalue grows  exponentially whereas the minimal eigenvalue is bounded. This situation results in an exponential growth of the condition number
 \[
 cond(A) = \mathcal{O}(N^2)=\mathcal{O}(2^{2k}),\quad \mbox{ for some $k>1$}.
 \]
 As a result, any iterative solver becomes slower, and a preconditioning is highly needed.
  To be precise for the Krylov subspace type methods, the solution of the linear system  $Au=b$ with some $u_0$ is
  \[
  \|u^{j}-u\|_A \le 2\left(\frac{\sqrt{\rho(A)}-1}{\sqrt{\rho(A)}+1}\right)^j \|u-u^0\|_A,  \ \|x\|_A = x^TAx,
  \]
 where $\rho(A)$ is the spectral condition number of $A$.  The convergence of the above expression is neat, but it has rarely been presented the convergence of conjugate gradient type methods unless $\rho(A)\approx 1$~\cite[page 128]{ke.Chen2005}. Thus it becomes clear that one needs to find a matrix $D$ such that $$B=D^{-1/2}A D^{-1/2}$$ is well conditioned. It is very popular to replace $\rho(A)$ in the iterative solvers by $\rho(B)$, which is called preconditioning~\cite{ke.Chen2005, CCStolk2011} and is used for the preconditioned linear system solvers. Thus we get the motivation to develop and to compare a few preconditioned solvers based on the established and popular preconditioning techniques for local second order elliptic operators.  We implement and demonstrate the power of multigrid, wavelet as well as Fourier preconditioners.
Our goal here is to implement  preconditioners in a traditional way so that
   \begin{enumerate}
   \item $D$ is a symmetric and positive definite matrix.

   \item $\rho(B) = \mathcal{O}(1)$, as $N\rightarrow \infty$.
   \end{enumerate}
To be specific, we discuss several types of preconditioners below. 
\begin{description}
\item[Wavelet Diagonal Preconditioning:]
One of the most successful preconditioners for elliptic PDEs is the wavelet diagonal preconditioning (WDP) which has been studied in details in~\cite{CCStolk2011, KUrban2009}, and many other references.
Since $\mathcal{L}$ is of elliptic type we attempt to implement wavelet diagonal preconditioning to solve \eqref{f:linearsystem_01}.  Suppose $\mathcal{L}$ is defined over a periodic domain. Then a preconditioner can be  defined  by combining  two separate steps:
\begin{enumerate}
    \item Define a basis transformation $\mathcal{F}$ (wavelet decomposition operator), given by a wavelet transformation, and a wavelet reconstruction operator $\mathcal{F}^*$ whose columns are the elements of the wavelet basis denoted by $\psi_\lambda$.

    \item Define an  invertible diagonal scaling matrix $\mathcal{S}$, whose elements are of the form $s_{\lambda}\approx 2^{-2|\lambda|}$, where $|\lambda|$ denotes the scale index of the wavelet.
\end{enumerate}

We consider the symmetric proconditioner $$ D=\mathcal{F}^*\mathcal{S}^{1/2}\mathcal{F}$$ as a scaled operator~\cite{CCStolk2011, KUrban2009}. Then we define the preconditioned operator (equivalent to $\mathcal{L}$ ) by
\[
\mathcal{F}^*\mathcal{S}^{1/2}\mathcal{F} \mathcal{L} \mathcal{F}^*\mathcal{S}^{1/2}\mathcal{F}.
\]
A detailed discussion about designing such a preconditioner can be found in~\cite{CCStolk2011}.
A preconditioner of this type is sensitive with boundary conditions.  One may also consider   $D=\mathcal{F}^*\mathcal{S}\mathcal{F}$ to define a left or a right preconditioner to implement a preconditioned  BICG solver. The implementation detail is same as the symmetric preconditioner discussed above.
\item[Fourier Sine Preconditioning:] Localization in the position-wave number space is an important concept in PDEs and can be extended to PIDEs. Most recently, a frame of functions, called windowed Fourier frames, has been employed to solve a variable coefficient second order elliptic PDE \cite{ccsskb2011}. Here it is our aim to design and to implement preconditioners based on the Fourier sine transformation (FSP) for the PIDE \eqref{f:linearsystem_01}. This preconditioning is sensitive with boundaries, and works very well for  periodic boundary value problems.

The symbol of the operator $\mathcal{L}$ defined by \eqref{f:generalmodel_dt_01} can be written as
 \begin{equation}\label{f:symbol02}
   \Delta t\sigma \xi^2 - \Delta t \mu i\xi + (1 + r\Delta t) - \Delta t \lambda  \sqrt{2\pi}(\hat J(\xi)-\hat J(0)).
 \end{equation}
When  $\xi$ is very large, $\xi^2$ term becomes the dominating term in \eqref{f:symbol02} and so $\mathcal{L}$ can be approximated by $\Delta t\sigma \xi^2$ in the frequency domain.
Thus  we approximate
\[
\mathcal{L}u \approx \dt \sigma \frac{\partial^2}{\partial x^2}u = \sum \dt\sigma \xi_k^2 b_{k} \sin (\xi_k x).
\]
Let $M_k=\dt\sigma \xi_k^2\ne 0$, then
\[
 \frac{1}{M_k} \dt \sigma\frac{\partial^2}{\partial x^2}u \approx \sum_{j\in \N}  b_{n}\sin(\xi_k x).
\]
Thus
\[
\left|\frac{1}{M_k} \dt\sigma\frac{\partial^2}{\partial x^2}u\right| \approx  \left|\sum_{j\in \N}  b_k \sin(\xi_k x) \right|.
\]
Now
\begin{eqnarray*}
\left|\sum_{j\in \N}  b_k \sin(\xi_k x)\right|^2 &\le& \sum_{j\in \N}  |b_k|^2 \le B \|u\|^2
\end{eqnarray*}
where $B$ is the frame upper bound~\cite{S.Mallat2009}.  It is clear from ~\cite{S.Mallat2009} that $B<\infty$ if we consider a tight frame. Thus
\[
 \left|\frac{1}{M_k}\dt \sigma\frac{\partial^2}{\partial x^2}u\right| <\infty,
\quad \text{and so}
\quad
 \left|\frac{1}{M_k}\mathcal{L}u\right| <\infty.
\]

Based on this idea we define a preconditioned operator $$P\mathcal{L}P,$$ with the symmetric preconditioner $$P=F^*M^{-1/2}F,$$ where $F$ stands for the Fourier sine transformation operator~\cite{S.Mallat2009}. The invertibility of the operator $P\mathcal{L}P$, considering $F$ a windowed Fourier transform operator, has been proved in \cite{ccsskb2011} (where $\mathcal{L}$ is a second order elliptic operator). The idea can be extended to the PIDE model that we have considered here. So we avoid attempting to prove the invertibility of  the Fourier sine preconditioner (FSP) $P\mathcal{L}P$ here.
One may also consider  $P=F^*M^{-1}F$ to define a left or a right preconditioner to implement a  preconditioned  BICG solver. The implementation detail is same as the symmetric precomnditioner we have discussed above.

\item[Multigrid Preconditioning:] Multigrid (MG) methods are now a days the fastest and most efficient numerical solvers for linear systems. There are huge recent literatures on MG methods. Actually multigrid method combines two separate ideas~\cite{ke.Chen2005, Cornelius.W.Oosterlee2001}:
    \begin{enumerate}
    \item fine grid residual smoothing by relaxation.
    \item coarse grid residual correction.
    \end{enumerate}
    Here the idea is to perform a few iterations (smoothing) in a fine grid, then switch to a coarser level and perform a few iterations, and so on.  This is called coarse grid corrections. After corrections, one switches back to the fine grid and performs a few post-smoothing.  Thus a multigrid algorithm uses three basic and old steps:
    \begin{itemize}
    \item relaxation step.
    \item restriction step.
    \item interpolation step.
    \end{itemize}
     A detailed discussion about multigrid can be found
    in~\cite{ L.Briggs, ke.Chen2005, Cornelius.W.Oosterlee2001} and in many other references. Since the operator \eqref{f:generalmodel_dt_01}  is of elliptic type, multigrid would be one of the choices to be considered to verify it's efficiency. Here we implement a so-called $v-$cycle to solve the system  (\ref{f:linearsystem_01}). It behaves well with both periodic and non-periodic boundary conditions.

\end{description}

In our problem we use just one $v$-cycle. One can use $\nu$ $v$-cycles if the solution is not
sufficiently accurate after the completion of one cycle.
We follow the Algorithm~\ref{mltgrd:alg} for computation.
\begin{algorithm}{\textbf{Multigrid method to solve system of linear equations}}\label{mltgrd:alg}
\\
To solve the system of linear equations $\mathcal{A}\underline u = \underline f$ using the multigrid method
\begin{description}
\item[INPUT] the finest grid matrix $\mathcal{A}^h,$ right-hand vector $\underline f^h,$ $L$ the number of steps to travel down the coarsest grid, $\mu$ the number of relaxation(iterations) on each grid, tolerance Tol, number of v-cycles $\nu,$ initial solution $\underline u^h = \underline 0$
\item[OUTPUT] The approximate solutions $\underline u^h.$
\item[Step 1] For $\gamma = 1, 2, \cdots, \nu$ or error $<$ Tol do the following steps
\item[Step 2] Relax $\mathcal{A}^h \underline u^h = \underline f^h$ $\mu$ times using the Jacobi iteration with the initial data $\underline u_0^h.$
\item[Step 3] Set $\underline r^h = \underline f^{h} - \mathcal{A}^h\underline u^h.$
\item[Step 4] For $k=2, 3, \cdots, L-1$\\
 define residual $\underline f^{kh} = \underline r^{kh}$,
where $r^{kh}_i=r^{(k-1)h}_{2 i-1}$, $i = 1, 2, \cdots, \frac{N}{k}$,
 take the initial guess $\underline u_0^{kh} = 0,$ and relax $\mathcal{A}^{kh} \underline u^{kh} = \underline f^{kh}$ $\mu$ times as in step 2.
\item[Step 5] Set $\underline f^{Lh} = \underline r^{Lh}$ and solve $\mathcal{A}^{Nh}\underline u^{Lh} =\underline f^{Lh}$ exactly.
\item[Step 6] For $k=L-1, L-2, \cdots, 1$
we upgrade $\underline u^{kh}$ by using
\begin{eqnarray*}
u_{2j-1}^{kh} &\longleftarrow& u_{2j-1}^{kh} + u_j^{(k+1)h};\quad j=1, 2, \cdots, N_{(k+1)h},\\
u_{2j}^{kh} &\longleftarrow& u_{2j}^{kh} + \frac{1}{2}
\left [
u_j^{(k+1)h} + u_{j+1}^{(k+1)h}
\right ];
\quad j=1, 2, \cdots, N_{(k+1)h}-1,\\
u_{2N_{(k+1) h}}^{kh} &\longleftarrow& u_{2N_{(k+1) h}}^{kh} + \frac{1}{2}
\left [
u_1^{(k+1)h} + u_{N_{(k+1) h}}^{(k+1)h}
\right ]
\end{eqnarray*}
 and upgrade the solution $\mu$ times as in step 2.
\item[Step 7] If $\|\underline u^h\|\le \mbox{Tol}$ or for
 $\nu > \gamma>1$ if $\|\underline u^{h,\gamma} - \underline u^{h,\gamma+1}\|\le \mbox{Tol},$ output the required solution $\underline u^h$
else  ``program stopped after $\nu$ v-cycle''.
\item[STOP]
\end{description}
\end{algorithm}
This is to note that one may use FFT for each matrix vector multiplication to reduce the computational costs since the operator  $A$ acting on \eqref{f:linearsystem_01} is a toeplitz matrix~\cite{Gol}.
\subsection{Numerical results and discussions}
Here we present some experimental/computer generated results to demonstrate the efficiency of the schemes.  We implement the schemes in MATLAB. The MATLAB function "FFT" is used to define the Fourier sine preconditioner; MATLAB functions "wavedec", and "waverec" have been used for the wavelet diagonal preconditioner  with  the Daubechies wavelet 'db6'. Here we consider a spatial periodic $[0 \ 1]$ domain and
$$J(x) = \sum_{r=-\infty}^{\infty} J^{\infty} (x-r)\quad \text{with} \quad J^{\infty}(x)=\sqrt{\frac{100}{\pi}}e^{-100x^2}.$$ A detailed discussion about such a consideration of the kernel function can be found in \cite{SamirKumarBhowmik04}.  We consider $\sigma=0.01$, $\mu=0.01$, $r=0.01$, $\lambda=0.1$, $\delta=100$ for all the numerical results presented here.

In Figure~\ref{f:fig01aaa}, we present condition numbers of the preconditioned operators $P\mathcal{L}P$, and $D\mathcal{L}D$, as well as the condition number of $\mathcal{L}$. We notice that $\rho(D\mathcal{L}D)$, and  $\rho(P\mathcal{L}P)$ are of $\mathcal{O}(1)$, where as $\rho(\mathcal{L})$ is of $\mathcal{O}(N^2)$.
Then, in Figure~\ref{f:fig02}, we compare the number of iterations taken by the preconditioned solvers for a set of $N$ values. We notice that the preconditioned systems converge in a few iterations and the number of iterations is independent of the system size.
\begin{figure}[here]
\begin{center}
\includegraphics[height=0.45\textwidth, width=.49\textwidth]{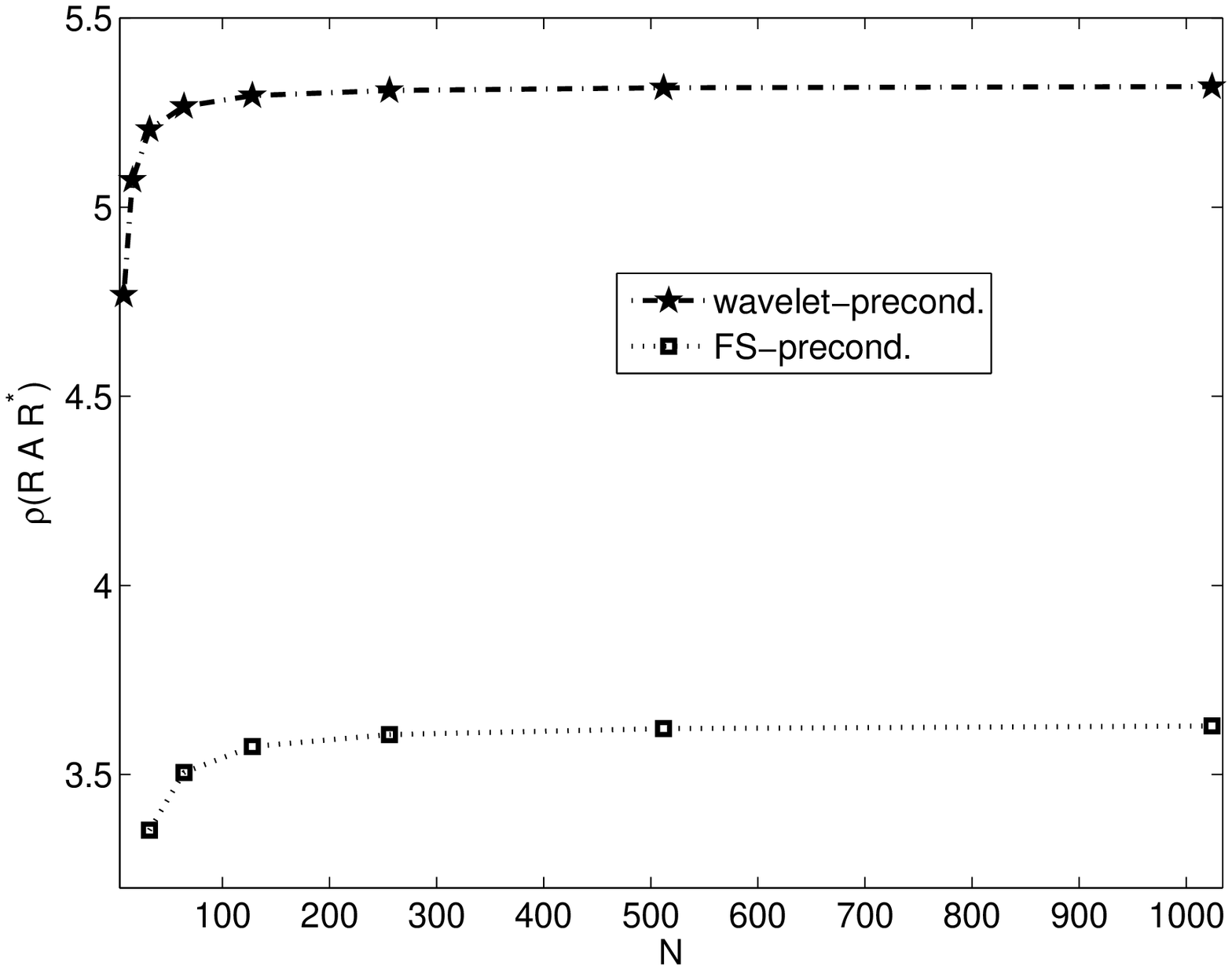}
\includegraphics[height=0.45\textwidth, width=.49\textwidth]{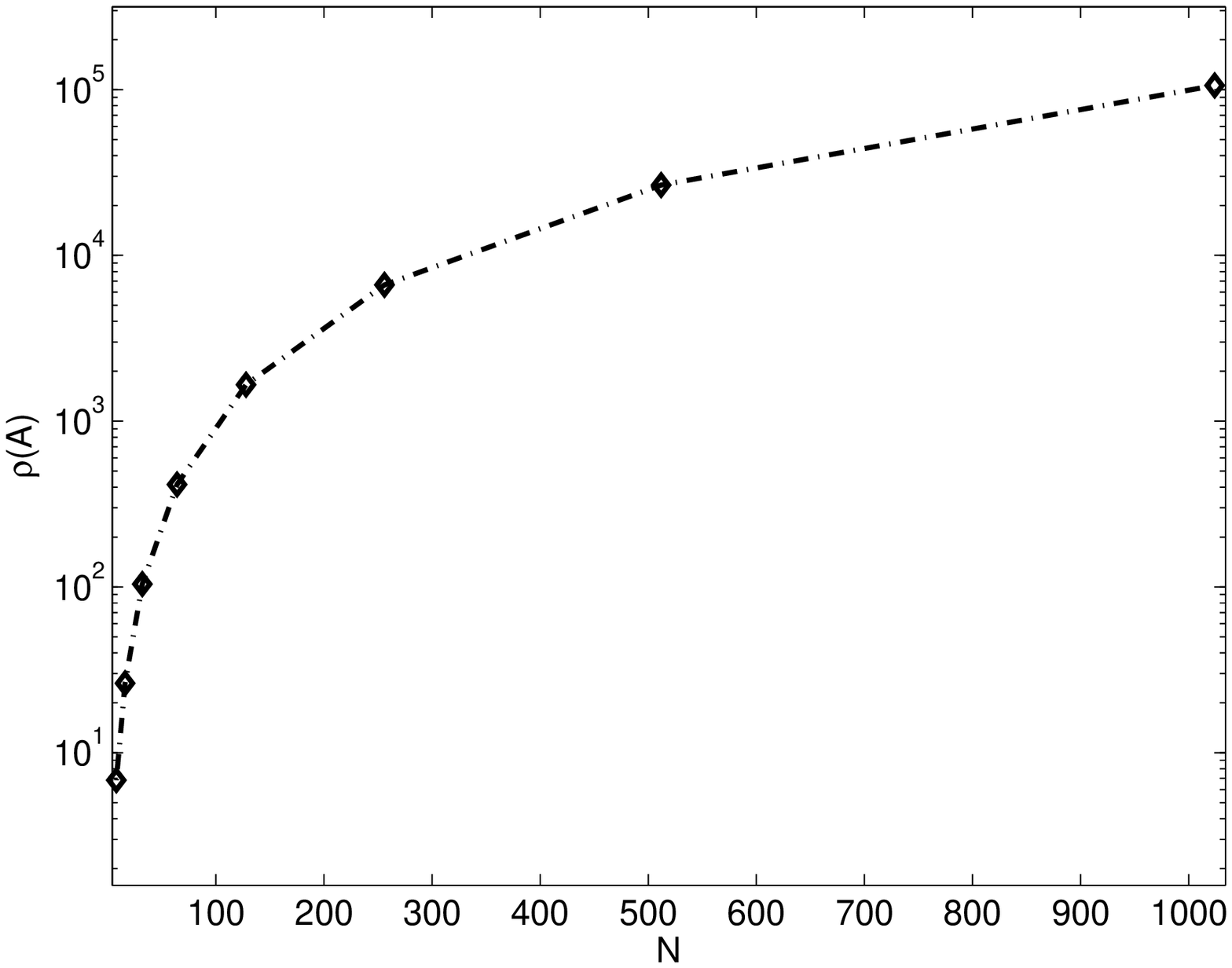}
\caption{Left Figure: Condition numbers of the wavelet preconditioned operator, and the Fourier sine preconditioned operator, both are of $\mathcal{O}(1)$,
Right Figure: Condition number of $A$, which is of $\mathcal{O}(N^2)$.}
\label{f:fig01aaa}
\end{center}
\end{figure}
\begin{figure}[t]
\begin{center}
\includegraphics[height=0.45\textwidth, width=.7\textwidth]{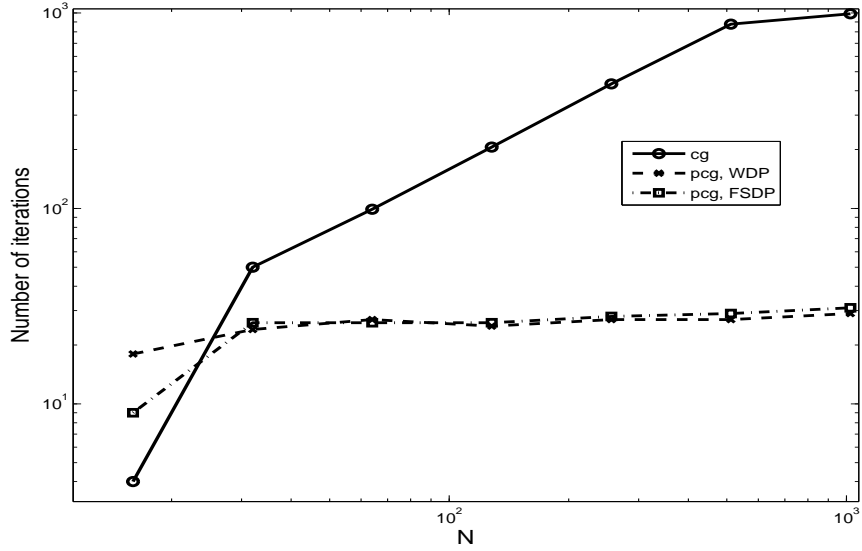}
\caption{The Number of iterations taken to converge by  the conjugate gradient, the WDP conjugate gradient and the FSP conjugate gradient methods to solve \eqref{f:linearsystem_01}  considering $\dt=0.01$.}
\label{f:fig02}
\end{center}
\end{figure}
Then we demonstrate the total CPU time taken to solve the linear system  by the solvers MG, WDP CG and FSP CG respectively to see the time efficiency of the techniques in Figure~\ref{f:time_mg_wdp_fsp1}. Here we observe that in terms of CPU time the MG out performs all other schemes. In fact, the MG method takes very little computational time compared to the other two.  The WDP and FSP techniques take most of the time to define the preconditioners, the preprocessing steps to use preconditioned linear system solvers.
\begin{figure}[here]
\begin{center}
\includegraphics[height=0.45\textwidth,width=.7\textwidth]{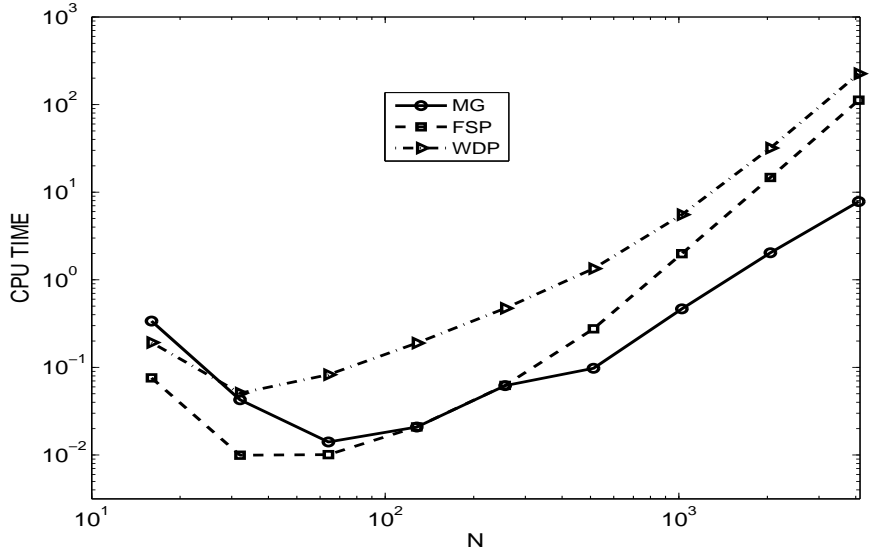}\\
\caption{CPU time taken to converge by the  WDP, the FSP and the MG methods for various choices of system size. For the multigrid method we consider \textbf{one $\log2(N)-2$ level $v-$cycle}  with one SOR iteration with $\omega =1.2$ in all levels but the coarsest one. In the coarsest level we solve the system \textbf{exactly}. We use Intel(R) Core(TM) i3 CPU  M380 at 2.53 GHz processor with ram 2.00 GB.
Here we solve \eqref{f:linearsystem_01} considering $\dt=0.01$, and by varying spatial grid points $N$, and $u_0(x) =\exp[-100(x-.5)^2]$.
}
\label{f:time_mg_wdp_fsp1}
\end{center}
\end{figure}
\subsection{An explicit implicit scheme}
While solving the linear system \eqref{f:linearsystem_01} we notice that $A$ is a full matrix. Thus matrix vector multiplications are computationally costly.  To reduce the computation cost further we look for an another scheme that may reduce computational costs.  We implement an explicit implicit scheme where $A$ becomes a spare matrix, thus reduces computation costs in matrix vector multiplications.

 We approximate the model \eqref{f:generalmodel01} in time by %
\[
 - \Delta t\sigma  \frac{\partial^2 u^n(x)}{\partial x^2} +(1 + r\Delta t)u^{n}(x) =u^{n-1}(x) + \Delta t \mu\frac{\partial u^{n-1}(x)}{\partial x} + \Delta t \lambda  \int_{\Omega}\ J(x-y)\left( u^{n-1}(y)-u^{n-1}(x)
   \right)dy,
\]
where $u^{n}(x)=u(x, t_n)$, $n \ge 0$. For simplicity we write
\begin{equation}\label{f:generalmodel_dt_02}
\mathcal{L}_1(u^n(x)) =\mathcal{L}_2 u^{n-1}(x),
\end{equation}
where
\[
 \mathcal{L}_1(u^n(x)) = -\Delta t\sigma  \frac{\partial^2 u^n(x)}{\partial x^2} +(1+r\Delta t)u^{n}(x),
\]
and
\[
\mathcal{L}_2(u^n(x)) =   \Delta t \mu\frac{\partial u^{n-1}(x)}{\partial x} + \Delta t \lambda  \int_{\Omega}\ J(x-y)\left( u^{n-1}(y)-u^{n-1}(x)
   \right)dy.
\]
The operator $\mathcal{L}_1$ is an elliptic partial differential operator~\cite{LCEvans1998}.
After the time integration,  the right hand side  of \eqref{f:generalmodel_dt_02} is a known vector and  explicitly depends on $u^{n-1}$, $n\ge 1$. Thus all the linear algebra tools we discussed above for \eqref{f:generalmodel_dt_01} are applicable to \eqref{f:generalmodel_dt_02}, and they are indeed, efficient  schemes for elliptic PDEs.

 To justify our claim we implement the MG method, the fastest tool we implemented in the previous section, to solve the linear system obtained from  \eqref{f:generalmodel_dt_02}. We compare the CPU time taken to solve the linear system obtained by the implicit solver and the explicit implicit solver \eqref{f:generalmodel_dt_02} in Figure~\ref{f:time_mg_imp_expimp_2013}.  Here we notice that the scheme \eqref{f:generalmodel_dt_01} and the explicit implicit scheme \eqref{f:generalmodel_dt_02} are comparable. In fact, it is observed from  Figure~\ref{f:time_mg_imp_expimp_2013} that the scheme \eqref{f:generalmodel_dt_02} requires a minimum  CPU time to converge compared to all other solvers.
\begin{figure}[here]
\begin{center}
\includegraphics[height=0.45\textwidth,width=.7\textwidth]{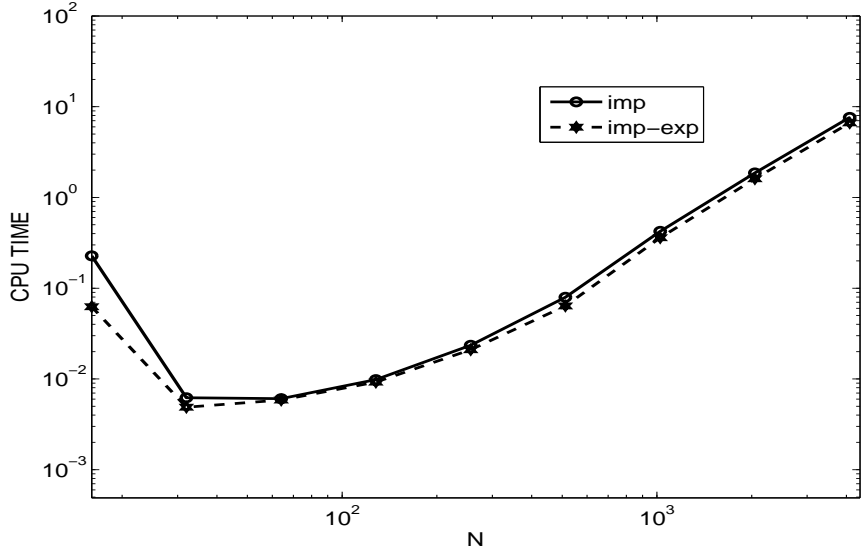}\\
\caption{CPU time taken to converge by MG methods for various choices of system size from the implicit as well as the explicit implicit solvers. Here we consider \textbf{one $\log2(N)-2$ level $v-$cycle} with one SOR iteration with $\omega =1.2$ in all levels but the coarsest one. In the coarsest level we solve the system \textbf{exactly}. We use Intel(R) Core(TM) i3 CPU  M380 at 2.53 GHz processor with ram 2.00 GB.
Here we solve \eqref{f:linearsystem_01}  and \eqref{f:generalmodel_dt_02} considering $\dt=0.01$, and by varying spatial grid points $N$, and $u_0(x) =\exp[-100(x-.5)^2]$.
}
\label{f:time_mg_imp_expimp_2013}
\end{center}
\end{figure}

\section{Stability analysis}\label{section02}
From the above Section we see that the scheme  \eqref{f:generalmodel_dt_02} dominates the implicit scheme in terms of computational time.  This numerical
 experiment motivates us to analyze the stability and the accuracy of an explicit and an explicit implicit scheme.  For the simplicity of the stability analysis we consider $\sigma = \mu=\lambda=r=1$.
Here we analyze the stability of the forward Euler scheme (explicit) and a mix Euler scheme (explicit implicit).
We consider the linear partial integro-differential equation~\cite{ D.J.Duffy03} (an IVP)
\begin{equation}\label{vna01:f}
u_{t}(x,t)=-u+u_x + u_{xx}+\int_{-\infty}^{\infty} J(x-y)\left( u(y,t)-u(x,t)\right) dy
\end{equation}
with $u(x,t_0)=u_0 (x),$  $x\in \R.$
This IVP  can be approximated  in space by
\begin{eqnarray}\label{vna02:f}
\frac{d\uh_{j}(t)}{dt}&=&
\frac{\uh_{j+1}-2\uh_j+\uh_{j-1}}{h^2}
+\frac{\uh_{j+1}-\uh_{j-1}}{2h}- \uh_j
\nonumber\\
&&\qquad+\qquad h \sum_{k=-\infty}^{\infty} J(x_j-x_k)(\uh_k-\uh_j)
\end{eqnarray}
for each $j\in \mathbb{Z}$ where $\uh_{j}(t)\approx u(x_j,t)$ and $x_j=jh$ where $h$ is the uniform spacing between the grid points $x_j$ and $x_{j+1}$ for all $j\in \mathbb{Z}$.
We need the following definitions to support our study.

For the
sequence $\{v_{m}:m\in \Z \}$ on the mesh points $\{x_{m} = mh : m\in \Z \}$ the discrete Fourier Transform (DFT) is defined by
\begin{equation}\label{vna04:f}
\tilde v(\xi)=\frac{h}{\sqrt{2 \pi}}\sum_{m=-\infty}^{\infty}
 e^{-i h m\xi} v_{m}
\end{equation}
if $v_m \in L_2 (h \Z)$,
and its inverse is %
\begin{equation}\label{vna05:f}
v_m=\frac{1}{\sqrt{2 \pi}}\int_{-\frac{\pi}{h}}^{\frac{\pi}{h}}
 e^{i h m\xi} \tilde v(\xi) d\xi 
\end{equation}
where $\xi \in[\frac{-\pi}{h}, \frac{\pi}{h}].$
Parseval's Formulae ~\cite{Str, J.S.Walker} are defined as
\begin{equation}\label{vna05pd:f}
|| \tilde v||_{h}^{2}=
\int_{-\frac{\pi}{h}}^{\frac{\pi}{h}}|\tilde v(\xi)|^2 d\xi=
\sum_{m=-\infty}^{\infty} h |v_m|^{2}=||v||_{h}^{2}.
\end{equation}
\subsection*{An explicit scheme}
We apply the explicit Euler scheme  to the semi-discrete model (\ref{vna02:f}) to obtain
\begin{eqnarray*}
\uh_{j}^{n+1}-\uh_{j}^{n}&=&
-\dt \uh_j^n+\dt \frac{U_{j+1}^n-U_{j-1}^n}{2h}+
\dt\frac{\uh_{j+1}^n-2\uh_j^n+\uh_{j-1}^n}{h^2}\\
&& \qquad+\qquad h \dt \sum_{k=-\infty}^{\infty} J(x_j-x_k)
\left( \uh_{k}^{n}-\uh_{j}^{n} \right)
\end{eqnarray*}
where $U_j^n = U(x_j, t_n).$
This is equivalent to
\begin{eqnarray}\label{vna03:f}
  \uh_{j}^{n+1}&=&
  \dt\frac{U_{j+1}^n-U_{j-1}^n}{2h}
  +
  \dt\frac{\uh_{j+1}^n-2\uh_j^n+\uh_{j-1}^n}{h^2} + \uh_{j}^{n}\left(1- \dt- h\dt  \sum_{k=-\infty}^{\infty} J(x_j-x_k)\right)
 \nonumber
 \\
 &&\qquad+
h\dt \sum_{k=-\infty}^{\infty} J(x_j-x_k)\uh_{k}^{n}.
\end{eqnarray}

%

%

We multiply (\ref{vna03:f}) by $\frac{h}{\sqrt{2\pi}}e^{-ijh\xi}$
and sum over all $j$ to obtain
\begin{eqnarray*}
\frac{h}{\sqrt{2\pi}}\sum_{j=-\infty}^{\infty} e^{-ijh\xi} \uh_{j}^{n+1}
&=&\frac{h}{\sqrt{2\pi}} \sum_{j=-\infty}^{\infty}e^{-ijh\xi} \uh_{j}^{n}\left(1-\dt-
 h\dt  \sum_{k=-\infty}^{\infty} J(x_j-x_k)\right)
\\
&&+
\frac{h}{\sqrt{2\pi}}\sum_{k=-\infty}^{\infty}e^{-ikh\xi} \uh_{k}^{n}
\left[ h\dt \sum_{j=-\infty}^{\infty} J(x_j-x_k)e^{-i(j-k)h\xi}\right]
\\
&&+
\frac{h}{\sqrt{2 \pi}}\sum_{j=-\infty}^{\infty}e^{-ijh\xi}\left(\dt\frac{\uh_{j+1}^n-2\uh_j^n+\uh_{j-1}^n}{h^2}\right)
\\
&&+
\frac{h}{\sqrt{2 \pi}}\sum_{j=-\infty}^{\infty}e^{-ijh\xi}\left(\dt\frac{\uh_{j+1}^n-\uh_{j-1}^n}{2h}\right).
\end{eqnarray*}
So using $J(x)=J(-x)$ we have
%
\begin{eqnarray*}
\tilde \uh^{n+1}(\xi)
& = &\left\{
1-\dt + h \dt \sum_{j=-\infty}^{\infty} J(x_j-x_k)
\left( e^{i(k-j)h\xi}-1\right)
\right\} \tilde \uh^{n}(\xi)
\\
&&+ \frac{\dt}{h^2} \tilde U^n(\xi) \left ( e^{ih\xi}+e^{-ih\xi} - 2 \right )
+
\frac{\dt}{2 h} \tilde U^n(\xi) \left ( e^{ih\xi}-e^{-ih\xi} \right).
\end{eqnarray*}
Thus
\begin{equation}\label{vna08:f}
\tilde \uh^{n}(\xi)= \left ( g(h\xi,\dt)\right )^{n} \tilde \uh^{0}(\xi),
\end{equation}
where
\begin{eqnarray}\label{vna07:f}
g(h \xi,\dt)&=&1-\dt + \dt\left(h \sum_{r=-\infty}^{\infty} e^{-irh\xi} J(x_r)
-
h \sum_{r=-\infty}^{\infty} e^{-irh 0} J(x_r) \right)
\nonumber \\
&& +\frac{\dt}{h^2}\left ( e^{ih\xi}+e^{-ih\xi} - 2 \right )+ \frac{\dt}{2h}\left ( e^{ih\xi}-e^{-ih\xi} \right )
\nonumber \\
&=&1-\dt+\sqrt{2 \pi} \dt\left(\tilde J(\xi)-\tilde J(0)\right)-4\frac{\dt}{h^2}\sin^2{\frac{h\xi}{2}}
+
\frac{i \dt}{h} \sin(h\xi).
\end{eqnarray}

%
Now we carry out the stability analysis of (\ref{vna03:f}) following~\cite{Q.Alfio, Str}.
We need the following  Lemma to bound $g(h\xi, \dt).$
\begin{theorem}\label{lemma01}
Assume that $J(x)\in L_2 (\mathbb{R})\cap C(\mathbb{R})$  satisfies
\begin{description}
\item [H1] $J(x) \ge 0;$
\item [H2]  $J(x)$ is normalized such that $\int_{-\infty}^{\infty} J(x)dx=1;$
\item [H3]  $J(x)$ is symmetric, i.e. $J(x)=J(-x),$ for all $x\in
\mathbb{R};$
\item [H4] $J(x)$ is decreasing on $(0, \infty);$
\item [H5] $\hat J(\xi) \ge 0.$
\end{description}
Then \textbf{H1}~-~\textbf{H4} give the DFT results
$0\le \tilde J(0)$ and
$\tilde J(\xi) \le \tilde J(0) \le \sqrt{\frac{2}{\pi}}+\tilde J(\xi)$
for all $\xi \in [-\frac{\pi}{h}, \frac{\pi}{h}]$ and the CFT results
$\hat J(\xi) \le \hat J(0) \le \sqrt{\frac{2}{\pi}}+\hat J(\xi).$
Further, if $\mathbf{H5}$ holds, then $\tilde J(\xi)\ge 0$ for all $J \in H^r(\mathbb{R})$, $r >\frac{1}{2}$, \cite{SKB02}.
\end{theorem}
%

Now we back to the main discussion. The scheme is stable if
\[
|g| \le 1.
\]
Here
\begin{eqnarray*}
|g|^2 &=& \left( 1-\dt+\sqrt{2 \pi} \dt(\tilde J(\xi)-\tilde J(0))-\frac{4\dt}{h^2}\sin^2\left(\frac{h\xi}{2}\right)\right )^2
\\
&&\qquad + \qquad
\left (\frac{\dt}{h}\right)^2 \sin^2 \left(h\xi \right) \le 1
\end{eqnarray*}
gives
\begin{eqnarray*}
 &&\dt\left(
1+ \tilde q^2(\xi) -2\tilde q(\xi) + \frac{8}{h^2}\sin^2\left(\frac{h\xi}{2}\right)
-\frac{8\tilde q(\xi)}{h^2}\sin^2\left(\frac{h\xi}{2}\right)+\right.\\
& & \left.
\frac{16}{h^4}\sin^4\left(\frac{h\xi}{2}\right)
+\frac{1}{h^2}\sin^2 h\xi \right)
 \le \left(2- 2\tilde q(\xi) +\frac{8}{h^2}\sin^2\frac{h\xi}{2} \right).
\end{eqnarray*}
Thus applying Proposition~\ref{lemma01} we have
\[
\dt \left(9+\frac{25}{h^2} + \frac{16}{h^4} \right) \le 4,
\]
and so
\begin{equation}\label{f:stabilitycondition01}
\dt \le \frac{4h^4}{ (3h^2+4)^2+h^2} \le  \frac{4h^4}{ (3h^2+4)^2} = \frac{4}{ (3 + 4/h^2)^2}.
\end{equation}
\begin{thrm}\label{lemma03}
If $J(x)$  is a normalized symmetric nonnegative function and $J \in L_2(\mathbb{R})\cap C(\mathbb{R})$ then
there exists $0 < \frac{4}{ (3 + 4/h^2)^2} \le \dt^{*}$  such that
 \[ \|\uh^n\|_h \le \|\uh^0\|_h\]
for all $0<\dt\le \dt^{*}$ and $n\ge 0.$
\end{thrm}
\begin{proof}
The proof easily follows from perseval's relation.
\end{proof}
Thus  in the discrete $L_2$ norm, (\ref{vna03:f}) is a stable  scheme~\cite[Definition 1.5.1]{Str} with the stability condition \eqref{f:stabilitycondition01}.

\subsection*{An explicit implicit scheme}
Applying a mixed Euler scheme we write a full discrete version of the model  (\ref{vna01:f}) by
\begin{eqnarray*}
\uh_{j}^{n+1}-\uh_{j}^{n}&=&
-\dt \uh_j^{n+1}+\dt \frac{U_{j+1}^n-U_{j}^n}{h}+
\dt\frac{\uh_{j+1}^{n+1}-2\uh_j^{n+1}+\uh_{j-1}^{n+1}}{h^2}\\
&& \qquad+\qquad h \dt \sum_{k=-\infty}^{\infty} J(x_j-x_k)
\left( \uh_{k}^{n}-\uh_{j}^{n} \right)
\end{eqnarray*}
where $U_j^n = U(x_j, t_n).$
This is equivalent to
\begin{eqnarray}\label{vna03aa:f}
  \uh_{j}^{n+1}\left(1+\dt\right)&=&
  \dt\frac{U_{j+1}^n-U_{j}^n}{2h}
  +
  \dt\frac{\uh_{j+1}^{n+1}-2\uh_j^{n+1}+\uh_{j-1}^{n+1}}{h^2} + \uh_{j}^{n}\left(1- h\dt  \sum_{k=-\infty}^{\infty} J(x_j-x_k)\right)
 \nonumber
 \\
 &&\qquad+
h\dt \sum_{k=-\infty}^{\infty} J(x_j-x_k)\uh_{k}^{n}.
\end{eqnarray}

%
Multiplying (\ref{vna03aa:f}) by $\frac{h}{\sqrt{2\pi}}e^{-ijh\xi}$
and summing over all $j$ we get
\begin{eqnarray*}
(1+\dt)\frac{h}{\sqrt{2\pi}}\sum_{j=-\infty}^{\infty} e^{-ijh\xi} \uh_{j}^{n+1}
&=&\frac{h}{\sqrt{2\pi}} \sum_{j=-\infty}^{\infty}e^{-ijh\xi} \uh_{j}^{n}\left(1-
 h\dt  \sum_{k=-\infty}^{\infty} J(x_j-x_k)\right)
\\
&&+
\frac{h}{\sqrt{2\pi}}\sum_{k=-\infty}^{\infty}e^{-ikh\xi} \uh_{k}^{n}
\left[ h\dt \sum_{j=-\infty}^{\infty} J(x_j-x_k)e^{-i(j-k)h\xi}\right]
\\
&&+
\frac{h}{\sqrt{2 \pi}}\sum_{j=-\infty}^{\infty}e^{-ijh\xi}\left(\dt\frac{\uh_{j+1}^{n+1}-2\uh_j^{n+1}+
\uh_{j-1}^{n+1}}{h^2}\right)
\\
&&+
\frac{h}{\sqrt{2 \pi}}\sum_{j=-\infty}^{\infty}e^{-ijh\xi}\left(\dt\frac{\uh_{j+1}^n-\uh_{j}^n}{h}\right).
\end{eqnarray*}
So using $J(x)=J(-x)$
%
\begin{eqnarray*}
\tilde \uh^{n+1}(\xi) (1+\dt-\frac{\dt}{h^2} \left ( e^{ih\xi}+e^{-ih\xi} - 2 \right ) )
& = &\left\{
1+ h \dt \sum_{j=-\infty}^{\infty} J(x_j-x_k)
\left( e^{i(k-j)h\xi}-1\right)
\right\} \tilde \uh^{n}(\xi)
\\
&&
+
\frac{\dt}{2 h} \tilde U^n(\xi) \left ( e^{ih\xi}-1 \right)
\end{eqnarray*}
giving
\[
\tilde \uh^{n+1}(\xi)=g(h\xi,\dt) \tilde \uh^{n}(\xi).
\]
And  we write
\begin{equation}\label{vna08a:f}
\tilde \uh^{n}(\xi)= \left ( g(h\xi,\dt)\right )^{n} \tilde \uh^{0}(\xi),
\end{equation}
where
\begin{eqnarray}\label{vna07:f}
g(h \xi,\dt)&=&
%
\frac{1+\sqrt{2 \pi} \dt\left(\tilde J(\xi)-\tilde J(0)\right)+ \frac{\dt}{h}\left ( e^{ih\xi}-1 \right )}
{1+\dt +4\frac{\dt}{h^2}\sin^2{\frac{h\xi}{2}}}.
\end{eqnarray}
The scheme is stable  if
\[
\left|1+\sqrt{2 \pi} \dt\left(\tilde J(\xi)-\tilde J(0)\right)+ \frac{\dt}{h}\left ( e^{ih\xi}-1 \right )\right| \le  \left|1+\dt +4\frac{\dt}{h^2}\sin^2{\frac{h\xi}{2}}\right|
\]
which gives
\[
\left(1+\sqrt{2 \pi} \dt\left(\tilde J(\xi)-\tilde J(0)\right)- \frac{\dt}{h} \right)^2+\left(\frac{\dt}{h}\right)^2 \le  \left(1+\dt +4\frac{\dt}{h^2}\sin^2{\frac{h\xi}{2}}\right)^2.
\]
Now
\[
 \left(1+\dt\right)^2\le  \left(1+\dt +4\frac{\dt}{h^2}\sin^2{\frac{h\xi}{2}}\right)^2,
\]
 and
 \[
 0 \le \left|\sqrt{2 \pi}\left(\tilde J(\xi)-\tilde J(0)\right)\right| \le 2.
 \]
Simplifying the above inequality we get
\[
\dt^2\left(\tilde q^2(\xi)-1 +\frac{2}{h^2}-\frac{2}{h}\tilde q(\xi) \right) \le \dt \left(2+\frac{2}{h}-2 \tilde q(\xi)\right),
\]
and so
\begin{equation}\label{f:stabilitycondition02}
\dt \le \frac{2h(h+1)}{3h^2+4h+2}.
\end{equation}

\begin{thrm}\label{lemma03a}
If $J(x)$  is a normalized symmetric nonnegative function and $J \in L_2(\mathbb{R})\cap C(\mathbb{R})$ then
there exists $0 < \frac{2h(h+1)}{3h^2+4h+2} \le \dt^{*}$  such that
 \[ \|\uh^n\|_h \le \|\uh^0\|_h\]
for all $0<\dt\le \dt^{*}$ and $n\ge 0.$
\end{thrm}
\begin{proof}
The proof easily follows from perseval's relation.
\end{proof}
Thus  in the discrete $L_2$ norm, (\ref{vna03aa:f}) is a stable  scheme~\cite[Definition 1.5.1]{Str} with the stability condition \eqref{f:stabilitycondition02}.
We demonstrate maximum values of $\dt$ from both 
 \eqref{f:stabilitycondition01} and \eqref{f:stabilitycondition02} respectively in Figure~\ref{f:dt_imp_exp_01} for various choices of $h$. It shows the dominance of the semi-implicit scheme.
\begin{figure}[here]
\begin{center}
\includegraphics[height=0.5\textwidth, width=.7\textwidth]{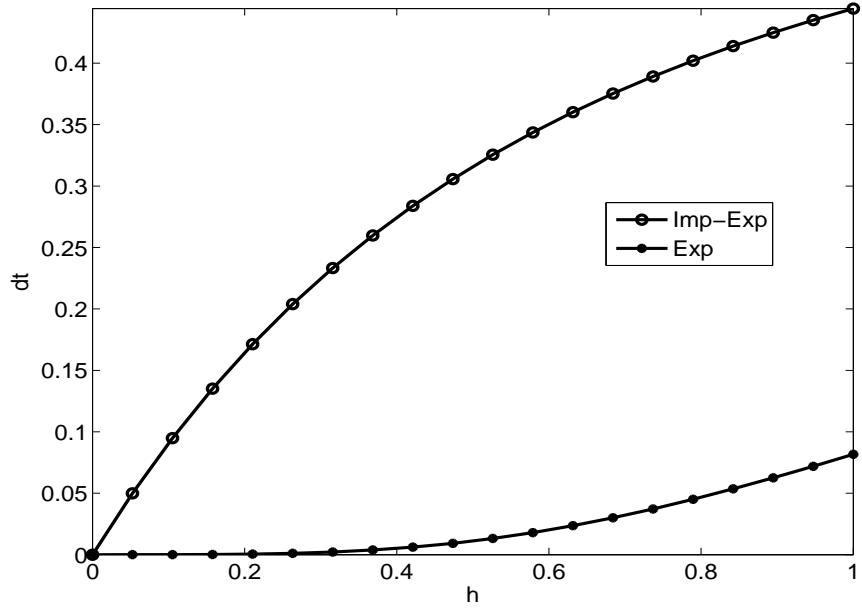}\\
\caption{Maximum choices of $\dt$ from the inequalities \eqref{f:stabilitycondition01} and \eqref{f:stabilitycondition02}.}
\end{center}
\label{f:dt_imp_exp_01}
\end{figure}
\subsection*{Computational algorithm}
From the schemes (\ref{vna08:f}) and (\ref{vna08a:f})  it follows that the DFT gives a each way to compute numerical solutions.
The approximate solution $\uh^n (\cdot)$ can be computed in the spatial domain simply, accurately and rapidly using the following steps.
For faster computations, one may precompute  the FFT of $u_0$, $J(x)$, and $J(0)$.
\begin{enumerate}

\item Compute the fast Fourier transform (FFT) of $u_0$.

\item Compute $g$  using the FFT of $J$.

\item Evaluate $g^n$ and multiply with the result in step~$1$.

\item Compute the inverse FFT of the product defined in step~$3$.

\end{enumerate}

\section{Accuracy analysis}\label{f:section04}
Applying the continuous Fourier transform  (\ref{vna01:f}) can be written as
\begin{equation}\label{vnaa01:f}
\hat u_{t}(\xi,t)=
\hat q(\xi)\hat u(\xi,t),
\end{equation}
where
\[
\hat q(\xi)=\sqrt{2\pi}\left(\frac{-1}{\sqrt{2\pi}} + \hat J(\xi)-\hat J(0)-\frac{\xi^2}{\sqrt{2\pi}}+ \frac{i\xi}{\sqrt{2\pi}}\right).
\]
Thus the exact solution of (\ref{vnaa01:f}) in the frequency domain is
\begin{equation}\label{vnaa09:f}
\hat u(\xi,t)=e^{\hat q(\xi)t} \hat u_{0}(\xi).
\end{equation}
Here it is easy to verify that  $\Re(\hat q) \le 0$ (Proposition~\ref{lemma01}, which is presented
in Section~\ref{section02})
which gives the stability property
$|\hat u (\xi, t)| \le |\hat u_0 (\xi, t)|.$
\subsection*{Computational algorithm}
The following steps can be taken to compute the exact solution and the error in schemes  (\ref{vna08:f}) and  (\ref{vna08a:f}).

\begin{enumerate}

\item Compute  the FFT of $u_0$.

\item Compute $\hat q$ as defined  in (\ref{vnaa01:f}) using FFT of $\hat J$.

\item  Evaluate $\exp (n\dt \hat q)$  and multiply with the result obtained from step $1$.

\item Compute the inverse FFT of the product defined in $3$.

\item Evaluate  $\|u(\cdot, t)-U^n(\cdot)\|$.
\end{enumerate}

In this section it is our aim  to present a
theoretical bound of the error term $\|u - U^n\|$. Now  we carry out the convergence analysis of (\ref{vna03:f}) and (\ref{vna03aa:f}) following~\cite{Str}.
We apply  the inverse CFT on (\ref{vnaa09:f}) to get
\begin{equation}\label{vnaa09_a:f}
u(x,t)=\frac{1}{\sqrt{2 \pi}} \int_{-\infty}^{\infty} e^{ix\xi}
e^{\hat q(\xi)t} \hat u_{0}(\xi) d\xi,
\end{equation}
which is the exact solution of (\ref{vna01:f}).
\subsection{The explicit scheme (\ref{vna08:f})}
Using the inverse DFT formula  (\ref{vna05:f}) on (\ref{vna08:f}), the approximate solution
can be presented as
\begin{equation}\label{vnaa07:f}
\uh_{m}^{n}=\frac{1}{\sqrt{2\pi}}\int_{-\frac{\pi}{h}}^{\frac{\pi}{h}}
e^{imh\xi} \left(g(h\xi, \dt)\right)^n \tilde u_{0}(\xi) d\xi.
\end{equation}
Applying  the Fourier  interpolation~\cite{Str} the mesh function \eqref{vnaa07:f} can be written as
\begin{equation}\label{vnaa08:f}
\mathcal{S} U^{n}(x)=\frac{1}{\sqrt{2\pi}}\int_{-\frac{\pi}{h}}^{\frac{\pi}{h}}
e^{ix\xi} \left(g(h\xi, \dt)\right)^n \tilde u_{0}(\xi)d\xi.
\end{equation}
Thus
\[
u(x, t_n)-\mathcal{S} U^n (x) = \frac{1}{\sqrt{2\pi}} \int_{|\xi|\le \frac{\pi}{h}}
e^{ix\xi}\left( e^{\hat q(\xi)t_n} \hat u_{0}(\xi)-\left(g(h\xi, \dt)\right)^n
\tilde u_{0}(\xi)\right)d\xi
\]
\begin{equation}\label{vnaa09a:f}
+ \frac{1}{\sqrt{2\pi}} \int_{|\xi|> \frac{\pi}{h}}
e^{ix\xi} e^{\hat q(\xi)t_n}\hat u_{0}(\xi)d\xi.
\end{equation}
So
\begin{eqnarray}\label{vnaa10:f}
\|u(x,t_n)-\mathcal{S}U^n (x)\|^2
&\le&\frac{1}{\sqrt{2\pi}} \int_{|\xi|\le \frac{\pi}{h}}
\left| e^{\hat q(\xi)t_n} \hat u_{0}(\xi)-\left(g(h\xi, \dt)\right)^n
\tilde u_{0}(\xi)\right| ^2 d\xi \nonumber\\
&& +\frac{1}{\sqrt{2\pi}} \int_{|\xi|> \frac{\pi}{h}}
 \left |\hat u_{0}(\xi) \right |^2 d\xi,
\end{eqnarray}
%
using Parseval's relation and the stability property $\hat q \le 0$.

%
Let us find a bound related to the-evolution error first. Here
\begin{eqnarray*}
&&\frac{1}{\sqrt{2\pi}} \int_{|\xi|\le \frac{\pi}{h}}
\left| e^{\hat q(\xi)t_n} \hat u_{0}(\xi)-g(h\xi,\dt)^n
\tilde u_{0}(\xi)\right| ^2 d\xi  \\
&\le& \sqrt{\frac{2}{\pi}} \int_{|\xi|\le \frac{\pi}{h}}
\left| e^{\hat q(\xi)t_n} -g(h\xi,\dt)^n
\right|^2|\hat u_{0}(\xi)| ^2 d\xi
 +\sqrt{\frac{2}{\pi}} \int_{|\xi|\le \frac{\pi}{h}}\left| \sum_{j\ne 0}
 \hat u_{0} \left(\xi+\frac{2\pi j}{h} \right) \right|^2 d\xi,
\end{eqnarray*}
since $\left|g(h\xi,\dt)\right|\le 1.$
Now following \cite[page 204]{Str},~\cite{SKB02}
\begin{eqnarray}\label{vnaa06a:f}
\sqrt{\frac{2}{\pi}}\int_{|\xi|\le\frac{\pi}{h}}\left|\sum_{j\ne 0}
\hat u_{0} \left(\xi+\frac{2\pi j}{h}\right)\right|^2 d\xi  
&\le& C_1(\sigma) h^{2\sigma}\|u_0\|_{H^{\sigma}(\R)}^{2},
\end{eqnarray}
where
$
C_1(\sigma) = 2\left(\frac{1}{\pi}\right)^{2\sigma}\sum_{j=1}^{\infty}
\left(2j-1\right)^{-2\sigma}
$
assuming  that the initial function is smooth and there exists
$\sigma>\frac{1}{2}$ such that $\|u_0\|_{H^{\sigma}(\mathbb{R})}$ is bounded,
and
\begin{eqnarray}\label{vnaa06b:f}
\frac{1}{\sqrt{2\pi}}\int_{|\xi|>\frac{\pi}{h}}|\hat u_{0}(\xi)|^2 d\xi
&\le& C_2(\sigma) h^{2\sigma}\|u_0\|_{H^{\sigma}(\R)}^{2}.
\end{eqnarray}

When $t_n=n\dt$
\[
e^{\hat q(\xi)t_n}-g(h\xi,\dt)^n
={e^{\hat q(\xi) \dt}}^n-g^n=(e^{\hat q(\xi) \dt}-g)\sum_{r=0}^{n-1}{e^{\hat q(\xi) \dt}}^{n-r}g^r.\]
Since $\hat q(\xi) \le 0$ and $|g(h\xi,\dt)| \le 1$
we have
\[
|{e^{\hat q(\xi) \dt}}^n-g^n|\le n|e^{\hat q(\xi) \dt}-g|,
\]
or equivalently
\begin{equation}\label{vnaa11:f}
\left |e^{\hat q(\xi)t_n}-{g(h\xi,\dt)}^n \right |\le n | e^{\hat q(\xi)\dt}-g(h\xi,\dt)|.
\end{equation}
Now, for the scheme (\ref{vna03:f}),
\begin{eqnarray}\label{vnaa02:f}
e^{\dt \hat q(\xi)}-g(h \xi, \dt)
&=& e^{\dt \sqrt{2 \pi} \left(\hat J(\xi)-\hat J(0)-\frac{\xi^2+1}{\sqrt{2\pi}}+\frac{i\xi}{\sqrt{2\pi}} \right) }-
\left(1-\dt + \dt\sqrt{2\pi}\left(\tilde J(\xi)-\tilde J(0)\right)\right.\nonumber\\
&& \left.-4\frac{\dt}{h^2}\sin^2\frac{h\xi}{2}+
\frac{i\dt}{h} sin(h\xi) \right)
\nonumber \\
&=&
\dt\sqrt{2\pi}\left(\hat J(\xi)-\hat J(0)-\frac{\xi^2+1}{\sqrt{2\pi}}+\frac{i\xi}{\sqrt{2\pi}}\right)
 -
\dt\sqrt{2\pi}\left(\tilde J(\xi)-\tilde J(0)\right. \nonumber\\
&&\left.- \frac{4}{\sqrt{2\pi} h^2}\sin^2\frac{h\xi}{2}
+\frac{i\dt}{h} sin(h\xi)-1 \right)\nonumber\\
&&
 + \sum_{j=2}^{\infty} \frac{\dt^j}{j!} \left(\sqrt{2 \pi}\left(\hat J(\xi)-\hat J(0)-\frac{\xi^2+1}{\sqrt{2\pi}}+\frac{i\xi}{\sqrt{2\pi}}\right)\right)^j.
\end{eqnarray}
Assuming that $J\in H^{r} (\mathbb{R})$
with $r>\frac{1}{2}$ and applying the Poisson summation formula,
(\ref{vnaa02:f}) becomes
\begin{eqnarray}\label{vnaa05:f}
e^{\dt \hat q(\xi)}-g(h \xi, \dt)
 &=& -\dt\sqrt{2\pi}\sum_{j\ne 0} \left( \hat J(\xi+\frac{2\pi j}{h})
-\hat J(\frac{2\pi j}{h})\right) + \frac{4 \dt}{h^2} \sin^2\left( \frac{h\xi}{2}\right)
-\dt \xi^2
\nonumber
\\
&& \qquad
 + i \dt \xi - \frac{i\dt}{h} \sin\left(h\xi\right)+ \mathcal{O}(\dt^2)\nonumber\\
&=&  -\dt\sqrt{2\pi}\sum_{j\ne 0} \left( \hat J(\xi+\frac{2\pi j}{h})
-\hat J(\frac{2\pi j}{h})\right) 
+ \mathcal{O}\left(\left(\frac{h\xi}{2}\right)^4 \right)\nonumber\\
&&
 + \frac{i\dt}{h} \left(h\xi\right)^3 + \mathcal{O}((h\xi)^5)+ \mathcal{O}(\dt^2).
\end{eqnarray}
%
%
\begin{theorem}\label{lemma04}~\cite{SKB02}
Assume that
\textbf{H1}, \textbf{H3} and \textbf{H5} of Lemma~\ref{lemma01} hold and
in addition, the following condition holds:
\begin{description}
\item [H6.]  $\frac{d}{d\xi} \hat J(\xi)\le 0$ for
$\xi\ge 0.$
\end{description}
Then, for all $|\xi|\le \frac{\pi}{h},$
$\qquad
\left|\sum_{j\ne 0} \left( \hat J(\xi+\frac{2\pi j}{h})
-\hat J(\frac{2\pi j}{h})\right)\right| \le 2 \hat J(\frac{\pi}{h}).
$
\end{theorem}
%
Thus applying Proposition~\ref{lemma04}, (\ref{vnaa05:f}) can be written as
\begin{eqnarray}\label{vnaa05b:f}
|e^{\dt \hat q(\xi)}-g(h \xi, \dt)| &\le& \dt C_1(h) + C_2\dt h^2 |\xi|^4 
\end{eqnarray}
where $C_1 (h)=2\sqrt{2\pi} \hat J (\frac{\pi}{h}).$
If $J\in L_2(\mathbb{R})$, then $|\hat J(\xi)|\rightarrow 0$ as $|\xi| \rightarrow \infty$
~\cite{K.Maleknejad},~\cite[page 30]{LNTre}. The rate of convergence determines the  accuracy of the scheme.
We have
\begin{eqnarray}\label{vnaa06c:f}
&&\int_{|\xi|\le \frac{\pi}{h}} \left|\left(e^{\hat q(\xi)t_n}-g(h\xi,\dt)^n \right)
\hat u_{0} (\xi)\right|^2 d\xi \nonumber \\
&\le&\int_{|\xi|\le \frac{\pi}{h}} n^2\left|e^{\hat q(\xi)\dt}-g(h\xi,\dt) \right|^2
|\hat u_{0} (\xi)|^2 d\xi, \quad \mbox {using (\ref{vnaa11:f})}\nonumber \\
&\le&  n^2  \int_{|\xi|\le\frac{\pi}{h}}
\left|\dt C_1(h) + C_2\dt h^2 |\xi|^3  \right|^2
|\hat u_{0}(\xi)|^2 d\xi, \quad \mbox {using (\ref{vnaa05b:f})}\nonumber \\
&\le&  t_n\int_{-\infty}^{\infty}\left|  C_1(h) + C_2 h^2 |\xi|^4  \right|^2
|\hat u_{0}(\xi)|^2 d\xi \nonumber \\
&\le& t_n C_1(h)\|u_0\|^2 +  t_n C_2 h^2\|u_0\|_{H^{2}(\R)}^2 .
\end{eqnarray}
Thus applying (\ref{vnaa06a:f}),(\ref{vnaa06b:f}) and (\ref{vnaa06c:f}), (\ref{vnaa10:f}) takes the form
\begin{equation}
\|u(x,t_n)-\mathcal{S}U^n (x)\|\le C_1(h)\|u_0\|^2 +  C_2 h \|u_0\|_{H^{2}(\R)}^2 +
C_3(\sigma) h^{\sigma}\|u_0\|_{H^{\sigma}(\R)}
\end{equation}
for all $u_{0}\in H^{\sigma}(\mathbb{R})$ with $\sigma>\frac{1}{2}.$
Thus we end up with the following result.
\begin{thrm}\label{thrmm:ff}
If  the kernel function $J(x)$ satisfies assumptions \textbf{H1}~-~\textbf{H6}
and (\ref{vna03:f})  is a  stable approximation for the IDE (\ref{vna01:f}),
  then there exist
constants $C_{1}(h),$ $C_2,$ $C_{3}(\sigma)$ such that
\[
\|u(x,t_n)-\mathcal{S}U^n (x)\|\le t_n C_1(h)\|u_0\| +  C_2 h \|u_0\|_{H^{2}(\R)}  +
C_3(\sigma) h^{\sigma}\|u_0\|_{H^{\sigma}(\R)}
\]
for any  $u_{0}\in H^{\sigma}(\mathbb{R})$ with $\sigma>\frac{1}{2}$.
\end{thrm}
%
%
\subsection{The explicit implicit scheme (\ref{vna03aa:f})}
Using series expansion
\begin{align*}
e^{\dt \sqrt{2 \pi} \left(\hat J(\xi)-\hat J(0)-\frac{\xi^2+1}{\sqrt{2\pi}}+\frac{i\xi}{\sqrt{2\pi}} \right) }
&= 1+ \dt \sqrt{2 \pi} \left(\hat J(\xi)-\hat J(0)-\frac{\xi^2+1}{\sqrt{2\pi}}+\frac{i\xi}{\sqrt{2\pi}} \right) \\
 &+ \dt^2 \left(\sqrt{2 \pi} \left(\hat J(\xi)-\hat J(0)-\frac{\xi^2+1}{\sqrt{2\pi}}+\frac{i\xi}{\sqrt{2\pi}} \right)\right)^2
 + \mathcal{O}(\dt^3).
\end{align*}
Also
\begin{eqnarray*}
g(h\xi, \dt) &=&  \left( 1+\sqrt{2 \pi} \dt\left(\tilde J(\xi)-\tilde J(0)\right)+ \frac{\dt}{h}\left ( e^{ih\xi}-1 \right )\right)\left(1+\dt +4\frac{\dt}{h^2}\sin^2{\frac{h\xi}{2}}\right)^{-1}.
\end{eqnarray*}
Letting  $\dt\le \frac{1}{1 +\frac{4}{h^2}\sin^2{\frac{h\xi}{2}}}=\frac{h^2}{h^2 +4\sin^2{\frac{h\xi}{2}}}\le 1$ (since $\min_{\theta} \sin^2\theta=0$), considering $\frac{\dt}{h}$, and $\frac{\dt}{h^2}$ constants we have
\[
\left(1+\dt +4\frac{\dt}{h^2}\sin^2{\frac{h\xi}{2}}\right)^{-1} = 1- \left(\dt +4\frac{\dt}{h^2}\sin^2{\frac{h\xi}{2}}\right) + \left(\dt +4\frac{\dt}{h^2}\sin^2{\frac{h\xi}{2}}\right)^2-\cdots,
\]
and so
\[
\dt\tilde q(\xi) \times \left(1+\dt +4\frac{\dt}{h^2}\sin^2{\frac{h\xi}{2}}\right)^{-1}= \dt\tilde q(\xi)  \left(1- \left(\dt +4\frac{\dt}{h^2}\sin^2{\frac{h\xi}{2}}\right) \right) + \mathcal{O}(\dt^3),
\]
where $\tilde q(\xi)=\sqrt{2 \pi}\left(\tilde J(\xi)-\tilde J(0)\right)$. Also
 \[
\frac{\dt}{h}\left( e^{ih\xi}-1 \right) =  \frac{\dt}{h} \left( {ih\xi} -h^2\xi^2/2 -\frac{i h^3\xi^3}{6}\right) + \mathcal{O}(h^4 \xi^4),
\]
gives
\begin{eqnarray*}
&&\frac{\dt}{h}\left( e^{ih\xi}-1 \right) \times \left(1+\dt +4\frac{\dt}{h^2}\sin^2{\frac{h\xi}{2}}\right)^{-1}\\
&=& \frac{\dt}{h} \left( {ih\xi} -h^2\xi^2/2 -\frac{i h^3\xi^3}{6}\right)  \left(1- \left(\dt +4\frac{\dt}{h^2}\sin^2{\frac{h\xi}{2}}\right) \right)\\
&=& \frac{\dt}{h} \left( {ih\xi} -h^2\xi^2/2 -\frac{i h^3\xi^3}{6}\right) - \frac{\dt^2}{h} \left( {ih\xi} -h^2\xi^2/2 -\frac{i h^3\xi^3}{6}\right)\\
&& - \left( {ih\xi} -h^2\xi^2/2 -\frac{i h^3\xi^3}{6}\right) 4\frac{\dt^2}{h^3}\sin^2{\frac{h\xi}{2}}.
\end{eqnarray*}
Thus
\begin{eqnarray*}
g(h\xi, \dt) &=& 1 + \dt \left[ -1-\frac{4}{h^2} \sin^2 \frac{h\xi}{2} + \tilde q(\xi)   +\frac{1}{h} \left( {ih\xi} -h^2\xi^2/2 -\frac{i h^3\xi^3}{6}\right)  \right]\\
&& +\dt^2
\left[
\left(1 +4\frac{1}{h^2}\sin^2{\frac{h\xi}{2}}\right)^2 -
\tilde q(\xi) \left(1 +4\frac{1}{h^2}\sin^2{\frac{h\xi}{2}}\right)\right. \\
&&-\left.
\frac{1}{h} \left( {ih\xi} -h^2\xi^2/2 -\frac{i h^3\xi^3}{6}\right)
\left(1 +4\frac{1}{h^2}\sin^2{\frac{h\xi}{2}}\right)
\right] + \mathcal{O}(\dt^3),
\end{eqnarray*}
gives
\[
\left|e^{\dt \sqrt{2 \pi} \left(\hat J(\xi)-\hat J(0)-\frac{\xi^2+1}{\sqrt{2\pi}}+\frac{i\xi}{\sqrt{2\pi}} \right) } - g(h\xi, \dt)\right| \le \dt C_1(h)+C_2\dt h^2\xi^4+C_3\dt^2.
\]
Thus following similar procedure as of the accuracy analysis of the explicit Euler scheme  we estimate the accuracy of the scheme  (\ref{vna03aa:f}) by the following theorem.
\begin{thrm}\label{thrmm1:ff}
If  the kernel function $J(x)$ satisfies assumptions \textbf{H1}~-~\textbf{H6}
and (\ref{vna03aa:f})  is a  stable approximation for the IDE (\ref{vna01:f}),
  then there exist
constants $C_{1}(h),$ $C_2,$ $C_3$, $C_{4}(\sigma)$ such that
\[
\|u(x,t_n)-\mathcal{S}U^n (x)\|\le t_n C_1(h)\|u_0\| +  C_2 h \|u_0\|_{H^{2}(\R)}  + C_3 \dt \|u_0\|
+ C_4(\sigma) h^{\sigma}\|u_0\|_{H^{\sigma}(\R)},
\]
for any  $u_{0}\in H^{\sigma}(\mathbb{R})$ with $\sigma>\frac{1}{2}$.
\end{thrm}
We compute error in such approximations that have been presented above considering various choices of the kernel function and the initial function. We present errors estimated by Theorem~\ref{thrmm:ff}  and Theorem~\ref{thrmm1:ff} in
Figure~\ref{f:fig01_err}. From this computation we observe the supremacy of the explicit implicit scheme as well as the importance of the choices of the initial function $u_0(x)$ and the kernel function $J (x)$. Here it can easily be noticed that smooth $J(x)$ and $u_0(x)$ give better accuracy and that justifies the
Theorem~\ref{thrmm:ff}  and the Theorem~\ref{thrmm1:ff}.
\begin{figure}[here]
\begin{center}
\includegraphics[height=0.35\textwidth, width=.49\textwidth]{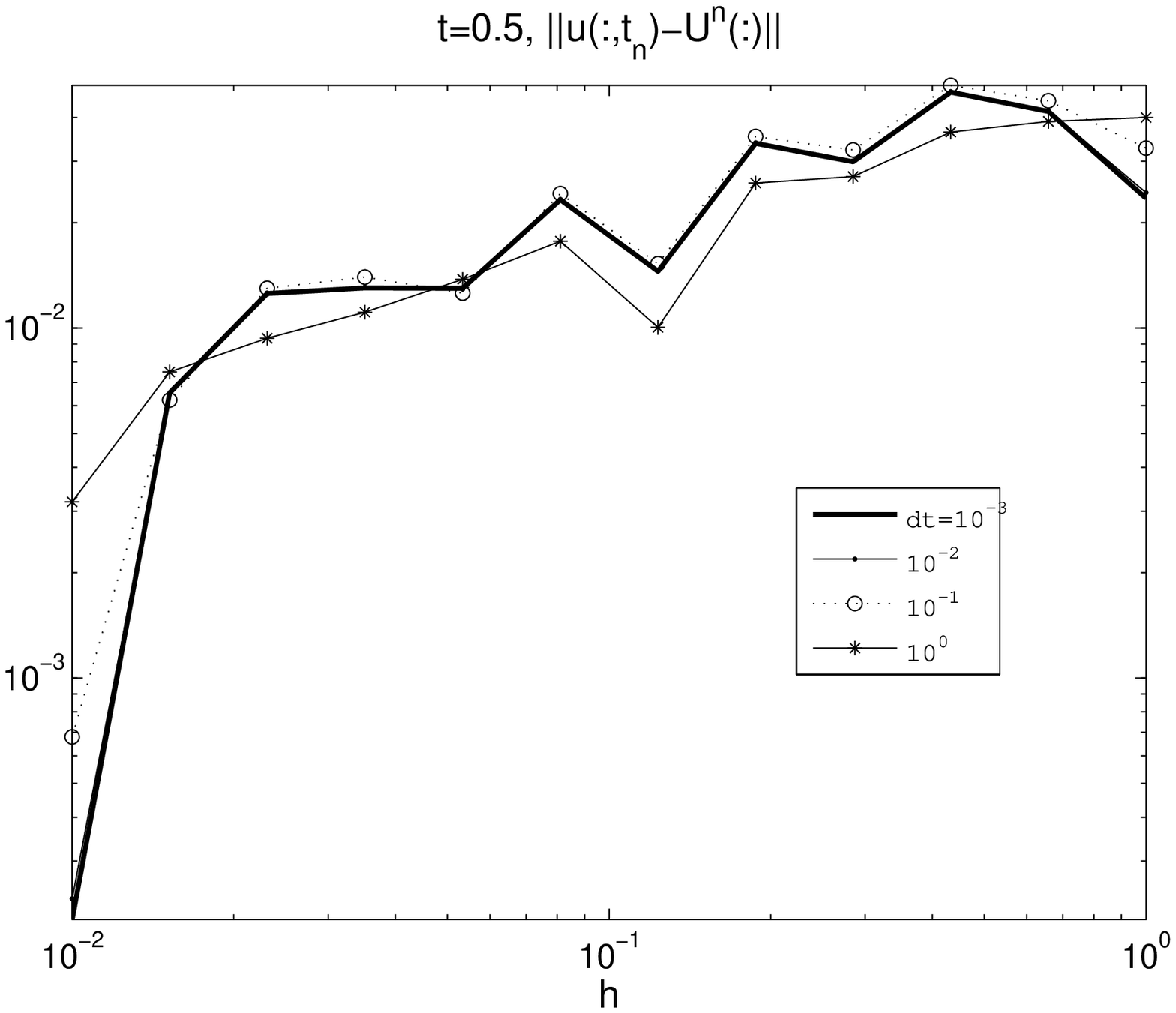}
\includegraphics[height=0.35\textwidth, width=.49\textwidth]{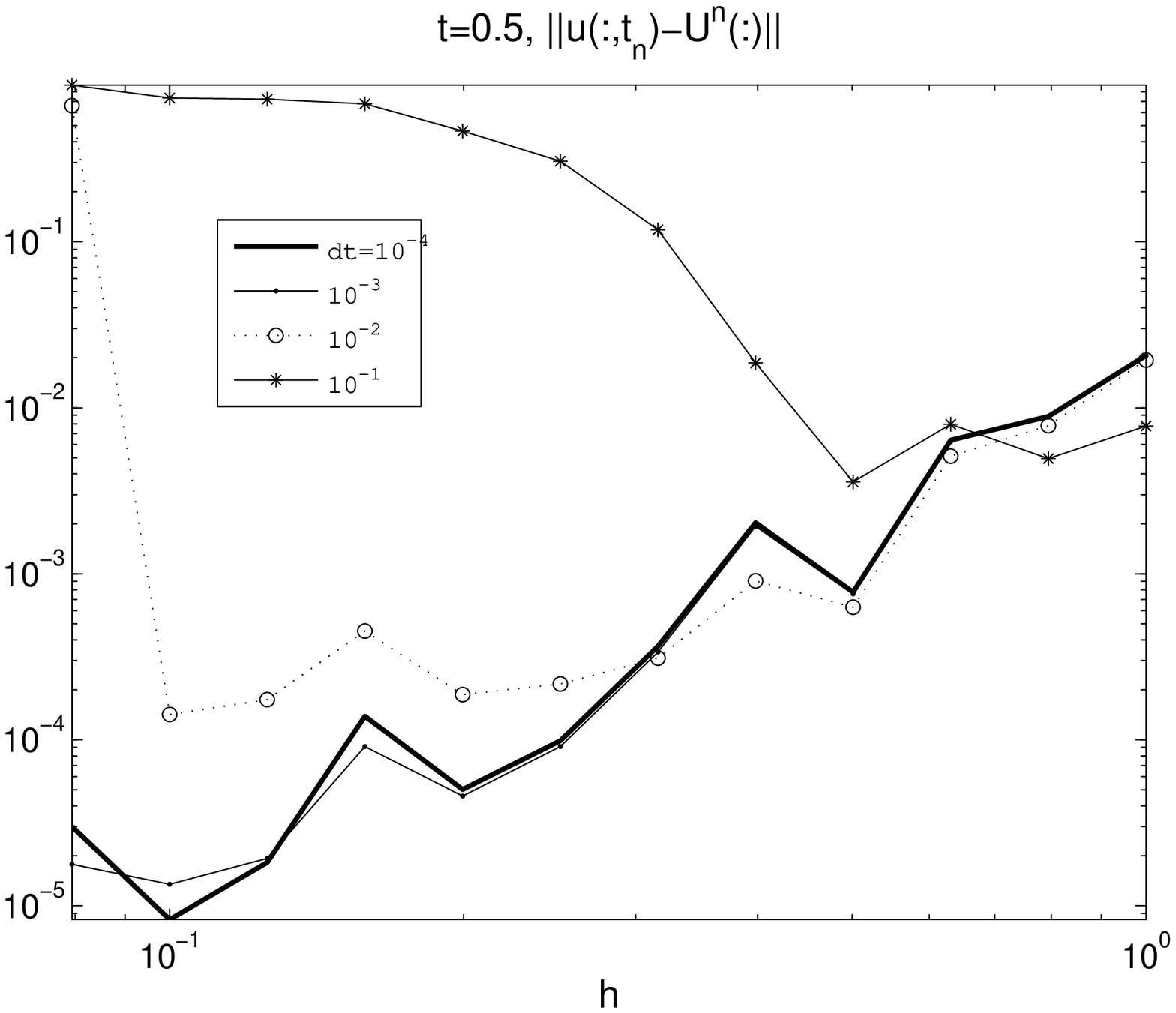}
\includegraphics[height=0.35\textwidth, width=.49\textwidth]{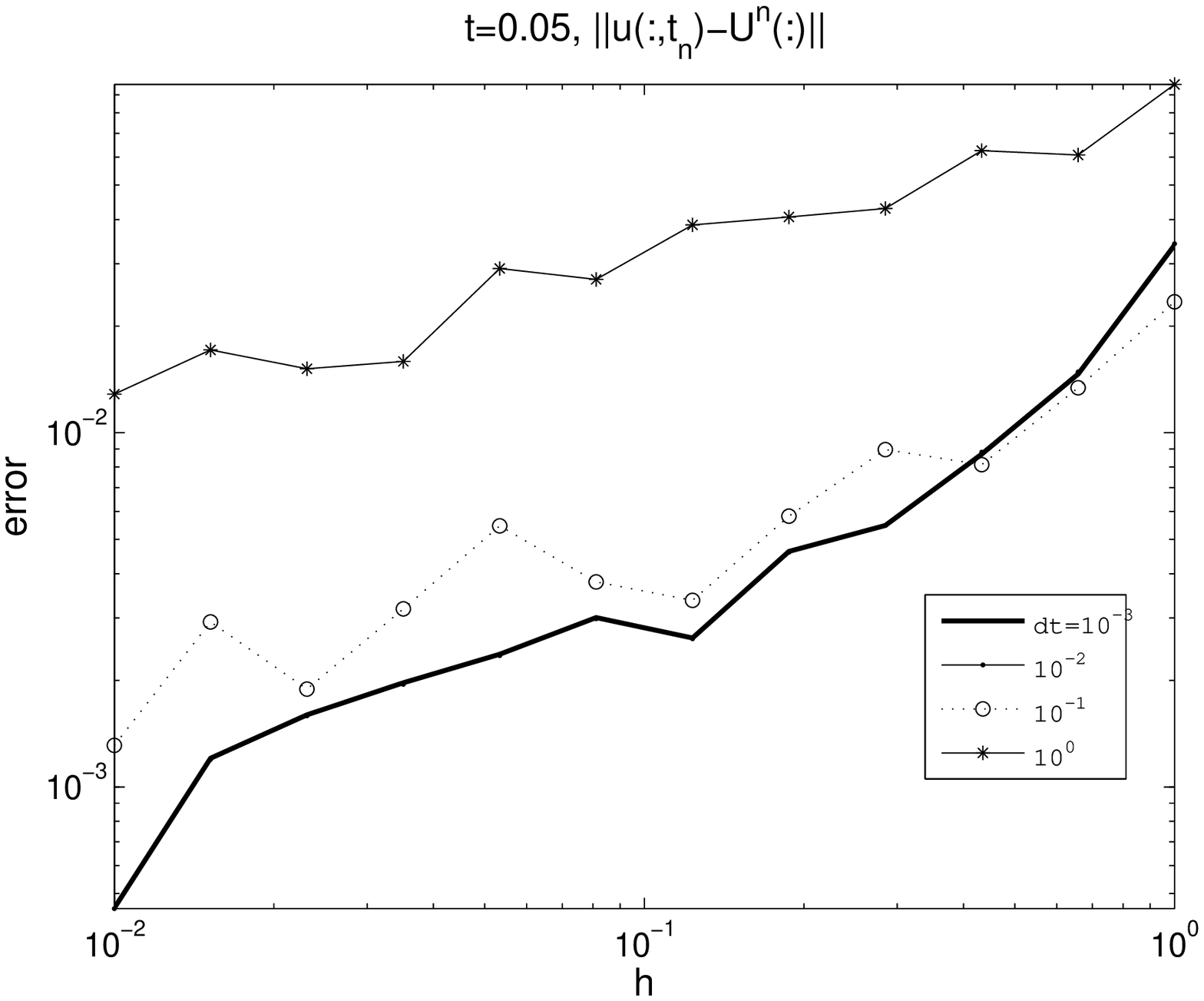}
\includegraphics[height=0.35\textwidth, width=.49\textwidth]{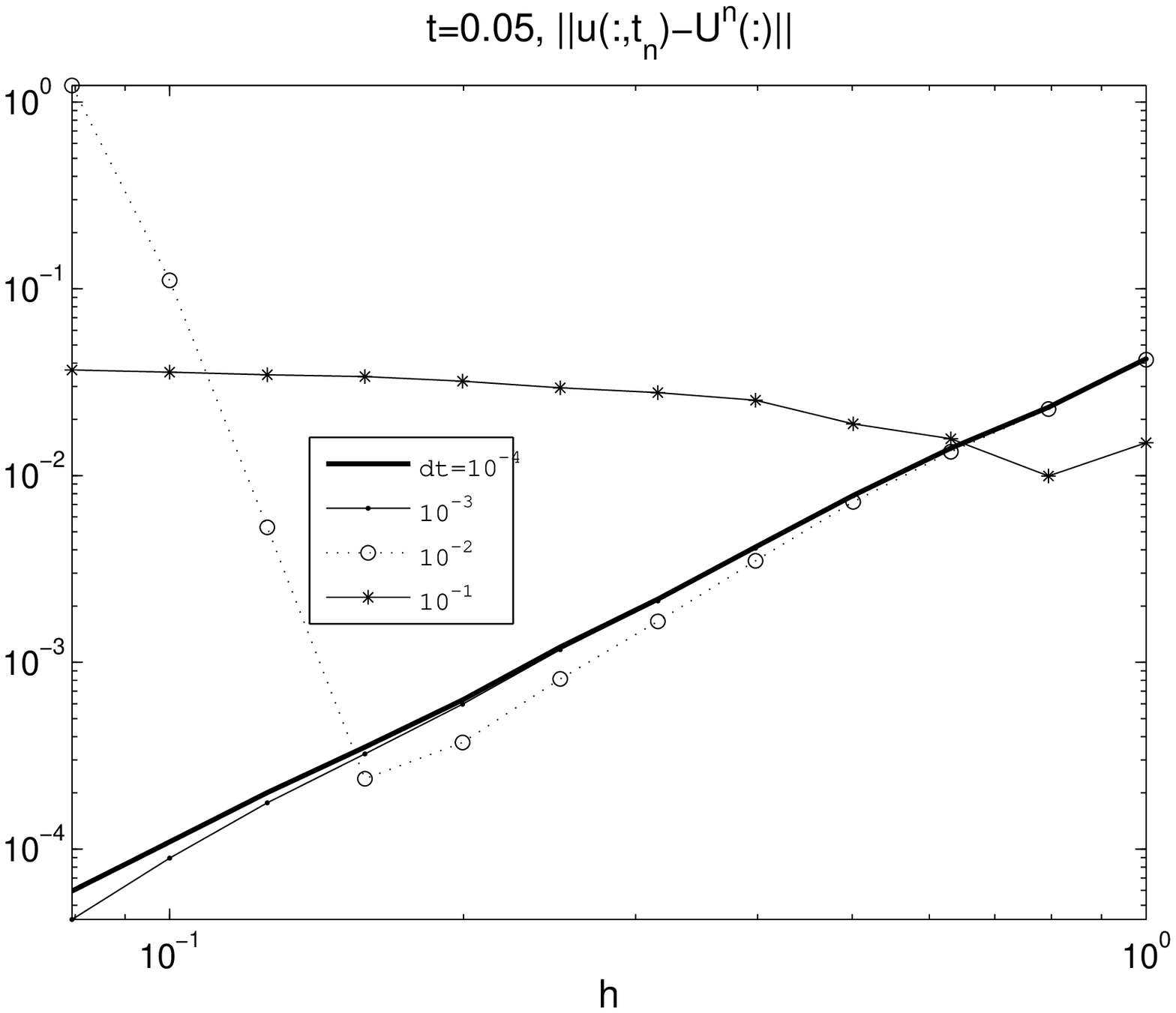}
\caption{Here we present the error $\|u(:, t_n) - U^n(:)\|$ estimated by the explicit scheme (right two) and the explicit-implicit (left two). Here in the  bottom Figures we consider $J(x)=e^{-|x|}$, $u_0(x)=e^{-|x|}$; in the top Figures we consider $u_0(x)=\sqrt{\frac{10}{\pi}}e^{-10x^2}$,    $J(x)=\sqrt{\frac{1}{\pi}}e^{-x^2}$
}
\label{f:fig01_err}
\end{center}
\end{figure}
\section{Accuracy of the semidiscrete approximation}\label{f:semidiscretesection}
%
Here we study the accuracy of the scheme (\ref{vna02:f}). Applying the discrete Fourier transform on (\ref{vna02:f})
\begin{equation}\label{asa01:f}
\tilde U_{t}(\xi,t)=\tilde q(\xi) \tilde U(\xi,t)
\end{equation}
where  $\xi\in \left[-\frac{\pi}{h},\frac{\pi}{h}\right]$ and
$$\tilde q(\xi)= \sqrt{2 \pi} \left( \tilde J(\xi)-\tilde J(0)\right) -\frac{4}{h^2}\sin^2{\frac{h\xi}{2}}+
\frac{i}{h}\sin({h\xi})-1$$
and thus 
\begin{equation}\label{asa02:f}
\tilde U(\xi,t)=e^{\tilde q(\xi)t} \tilde U_{0}(\xi).
\end{equation}
Applying the inverse Fourier transform to (\ref{asa02:f}) \begin{equation}\label{asa02a:f}
U_{m}(t)=\frac{1}{\sqrt{2\pi}} \int_{-\frac{\pi}{h}}^{\frac{\pi}{h}}
e^{imh\xi} e^{\tilde q(\xi)t}\tilde U_{0}(\xi) d\xi.
\end{equation}
We interpolate $U_{m}(t)$ defined in (\ref{asa02a:f}) by~\cite{Str}
\[
\mathcal{S} U(x,t)=\frac{1}{\sqrt{2\pi}} \int_{-\frac{\pi}{h}}^{\frac{\pi}{h}}
e^{ix\xi} e^{\tilde q(\xi)t}\tilde U_{0}(\xi) d\xi.
\]
Similar to the Theorem~\ref{thrmm:ff} (using (\ref{vnaa09_a:f})),
\[
u(x,t)-\mathcal{S} U(x,t)
=\frac{1}{\sqrt{2\pi}} \int_{|\xi|\le \frac{\pi}{h}}
e^{ix\xi}\left( e^{\hat q(\xi)t} \hat u_{0}(\xi)-e^{\tilde q(\xi)t}
\tilde U_{0}(\xi)\right)d\xi
\]
\begin{equation}\label{asa03:f}
+\frac{1}{\sqrt{2\pi}} \int_{|\xi|\ge \frac{\pi}{h}}
e^{ix\xi} e^{\hat q(\xi)t}\hat u_{0}(\xi)d\xi.
\end{equation}

\begin{thrm}
If
$J$, and  $\hat J$ satisfy the assumptions \textbf{H1}~-~\textbf{H6}
 and $u_{0}\in H^{\sigma}(\mathbb{R})$ with $\sigma>\frac{1}{2},$ then there exist
constants $C_{1}(h),$ $C_{3}(\sigma)$ such that
\[
\|u(x,t)-\mathcal{S} U (t)\|\le t C_1 (h) \|u_0\| +C_2 h \|u_0\|_{H^{2}(\R)}^2  +
C_3(\sigma) h^{\sigma}\|u_0\|_{H^{\sigma}(\R)},
\]
where (\ref{vna02:f}) is a semidiscrete approximation to  the IDE (\ref{vna01:f}).
\end{thrm}
\begin{proof}
We have
\begin{eqnarray}\label{asa04:f}
\|u(\cdot,t)-\mathcal{S} U(\cdot,t)\|^2&=&
\int_{-\infty}^{\infty} {\left|u(\cdot,t)-\mathcal{S}U(\cdot,t)\right|}^2 dx
\nonumber\\
&\le&\frac{1}{\sqrt{2\pi}} \int_{|\xi|\le \frac{\pi}{h}}
\left| e^{\hat q(\xi)t} \hat u_{0}(\xi)-e^{\tilde q(\xi)t}
\tilde U_{0}(\xi)\right| ^2 d\xi \nonumber\\
&& +\frac{1}{\sqrt{2\pi}} \int_{|\xi|\ge \frac{\pi}{h}}
 \left |\hat u_{0}(\xi) \right |^2 d\xi
\end{eqnarray}
since by Lemma~\ref{lemma01} $Real(\hat q) \le 0.$ Similar to
the analysis of the full discrete approximation
the first part
of the right-hand side of (\ref{asa04:f}) can be
written as
\begin{eqnarray*}
&&
\frac{1}{\sqrt{2\pi}} \int_{|\xi|\le \frac{\pi}{h}}
\left| e^{\hat q(\xi)t} \hat u_{0}(\xi)-e^{\tilde q(\xi)t}
\tilde U_{0}(\xi)\right| ^2 d\xi  \\
&\le& \frac{2}{\sqrt{2\pi}} \int_{|\xi|\le \frac{\pi}{h}}
\left| e^{\hat q(\xi)t} -e^{\tilde q(\xi)t}
\right|^2|\hat u_{0}(\xi)| ^2 d\xi 
+\frac{2}{\sqrt{2\pi}} \int_{|\xi|\le \frac{\pi}{h}}\left| \sum_{j\ne 0}
 \hat u_{0} \left(\xi+\frac{2\pi j}{h} \right) \right|^2 d\xi.
\end{eqnarray*}
We have
$$
\left|e^{\hat q(\xi)t}-e^{\tilde q(\xi)t} \right|
\le t|\hat q(\xi)-\tilde q(\xi)|
$$
since $real(\tilde q(\xi))\le 0 \quad \mbox{and}\quad real(\hat q(\xi))\le 0.$
Now
\begin{eqnarray*}
\hat q(\xi)-\tilde q(\xi)
&=& \sqrt{2\pi}\sum_{j=-\infty,j \ne 0}^{\infty}\left(\hat J \left(\xi+\frac{2\pi j}{h} \right)
-\hat J\left(\frac{2\pi j}{h}\right)\right)\\
&&+ \left( \xi^2 - \frac{4}{h^2}\sin^2(\frac{h\xi}{2}) \right)
+\left(-i\xi+\frac{i\sin(h\xi)}{h}\right)\\
&=& \sqrt{2\pi}\sum_{j=-\infty,j \ne 0}^{\infty}\left(\hat J \left(\xi+\frac{2\pi j}{h} \right)
-\hat J\left(\frac{2\pi j}{h}\right)\right)
+ C_2 h^2\xi^4 + \mathcal{O}((h\xi)^5).
\end{eqnarray*}
Thus
\begin{eqnarray*}
|\hat q(\xi)-\tilde q(\xi)|
&\le & C(h) + C_2 h^2\xi^4,
\end{eqnarray*}
and  $ C(h)= 2 \hat J(\frac{\pi}{h}) $ as $h\rightarrow 0$.
Now
\begin{align*}
\int_{|\xi|\le \frac{\pi}{h}}
\left| e^{\hat q(\xi)t} -e^{\tilde q(\xi)t}
\right|^2|\hat u_{0}(\xi)| ^2 d\xi
&\le t^2\int_{|\xi|\le\frac{\pi}{h}}\left|\hat q(\xi)-\tilde q(\xi) \right|^2
|\hat u_{0}(\xi)|^2d\xi \\
&\le t^2 C(h)^2 \| u_0\|^{2}
+ C^2 h^2\| u_0\|_
{H^2(\R)}^2,
\end{align*}%
%
gives
\begin{equation}\label{asa05:f}
\int_{|\xi|\le \frac{\pi}{h}}\left| e^{\hat q(\xi)t}-e^{\tilde q(\xi)t}
\right|^2 |\hat u_0(\xi)|^2 d\xi
\le t^2 C^2(h) \| u_0\|^{2} + C_2 t^2 h^2 \|u_0\|_{H^{2}(\R)}^2.
\end{equation}
Thus applying (\ref{vnaa06a:f}),(\ref{vnaa06b:f}) and (\ref{asa05:f}),
(\ref{asa04:f}) takes the form
\begin{equation}
\|u(x,t)-\mathcal{S} U (t)\|\le t C_1 (h) \|u_0\| +C_2 h \|u_0\|_{H^{2}(\R)}^2  + C_3(\sigma) h^{\sigma}\|u_0\|_{H^{\sigma}(\R)}
\end{equation}
for some $C_1$, $C_2$, $C_3$ for all $u_0 \in H^{\sigma}(\mathbb{R})$ with $\sigma>\frac{1}{2}.$
\end{proof}
\section{Summary and conclusions}\label{section03}
In this study, we consider a linear partial integro-differential operator (PIDO) that comes in modeling financial engineering problems as well as in modeling various scientific problems. We study a few  finite difference schemes (FDSs) for  European style options with a jump-diffusion term (the PIDO). In the first part of the study we introduce several preconditioned linear system solvers for the full discrete equivalent of the model. We observe that all the preconditioned solvers are very efficient, and the multigrid solver is way better than the wavelet diagonal preconditioned solver and the Fourier sine preconditioned solvers. In fact, a one $v-$cycled Multigrid solver is several times faster than the other two. The implementation costs for the sine and the wavelet preconditioning are relatively higher than that of the multigrid technique. So we conclude that a multigrid method can be used to speed up the computation of the finite dimensional (full discrete) PIDE model.  Here we also conclude that the explicit implicit scheme outperforms the implicit scheme in terms of computational costs.

Here, in the second part of this study, we analyze the stability and the accuracy of two different finite difference schemes. While analyzing the stability and the accuracy of the finite difference schemes (an explicit scheme as well as an explicit implicit scheme) we notice that the schemes are conditionally stable (under some reasonable restrictions imposed on the kernel function). The explicit implicit scheme is faster than that of the explicit scheme as well as the implicit scheme, which agrees with the properties of the time and the space discretizations of the PIDE we consider in this study. We establish some bounds of the  error in such full discrete as well as semi-discrete schemes.

Here we analyze the model in one space dimension only.   Preconditioners can be employed  to speed up the computational process for the full discrete model, specially for two and three space dimensional domains as well as preconditioned solvers along with higher order multi-step schemes may be better options to think of, and that leaves as future research directions.
\bibliography{ref_black_scholes}{}

\begin{thebibliography}{10}

\bibitem{Q.Alfio}
Q.~Alfio, R.~Sacco, and F.~Saleri.
\newblock {\em {Numerical Mathematics}}.
\newblock Springer, 2000.

\bibitem{SKB02}
S.~K. Bhowmik.
\newblock {\em Numerical approximation of a nonlinear partial
  integro-differential equation}.
\newblock PhD thesis, Heriot-Watt University, Edinburgh, UK, April, 2008.

\bibitem{SamirKumarBhowmik04}
S.~K. Bhowmik.
\newblock Numerical approximation of a convolution model of dot theta-neuron
  networks.
\newblock {\em Applied Numerical Mathematics}, 61:581--592, 2011.

\bibitem{ccsskb2011}
S.~K. Bhowmik and C.~C. Stolk.
\newblock Preconditioners based on windowed fouerier frames applied to elliptic
  partial differential equations.
\newblock {\em Journal of Pseudo-Differential Operators and Applications},
  2(3):317--342, April 2011.

\bibitem{L.Briggs}
W.~L. Briggs.
\newblock {\em {Multigrid Tutorial}}.
\newblock SIAM, Pennsylvania, 1987.

\bibitem{ke.Chen2005}
K.~Chen.
\newblock {\em {Matrix Preconditioning Techniques and Applications}}.
\newblock Cambridge University Press, 2005.

\bibitem{rcontevoltchkova2005}
R.~Cont and E.~Voltchkova.
\newblock A finite difference scheme for option pricing in jump diffusion and
  exponential levy models.
\newblock {\em SIAM J. NUMER. ANAL.}, 43(4):1596--1626, 2005.

\bibitem{Gab01}
S.~Coombes, G.~J. Lord, and M.~R. Owen.
\newblock Waves and bumps in neuronal networks with axo-dendritic synaptic
  interactions.
\newblock {\em SIAM Journal.}, 3(October), 2002.

\bibitem{D.J.Duffy03}
D.~J. Duffy.
\newblock {\em {Finite Difference Methods for Financial Engineering, A Partial
  Differential Equation Approach}}.
\newblock Wiley Finance, John Wiley and Sons, 2006.

\bibitem{Dug}
D.~B. Duncan, M.~Grinfeld, and I.~Stoleriu.
\newblock {Coarsening in an integro-differential model of phase transitions}.
\newblock {\em Euro. Journal of Applied Mathematics}, 11:511--523, 2000.

\bibitem{LCEvans1998}
L.~C. Evans.
\newblock {\em {Partial Differential Equations}}.
\newblock AMS, 1998.

\bibitem{F.Fiorani}
F.~Fiorani.
\newblock {\em {Option Pricing Under the Variance Gamma Process}}.
\newblock PhD thesis, University of Trieste, 2009.

\bibitem{Gol}
G.~H. Golub and C.~F.~V. Loan.
\newblock {\em {Matrix Computations}}.
\newblock Third edition, The Johns Hopkins University Press, Baltimore and
  London, 1996.

\bibitem{C.Chow}
Y.~Guo and C.~C. Chow.
\newblock Existence and stability of standing pulses in neural
  networks:|existence.
\newblock {\em SIAM J. Applied Dynamical Systems}, 4(2):217--248, 2005.

\bibitem{K.Maleknejad}
K.~Maleknejad.
\newblock A comparison of fourier extrapolation methods for numerical solution
  of deconvolution.
\newblock {\em New York Journal of Mathematics}, 183:533--538, 2006.

\bibitem{S.Mallat2009}
S.~Mallat.
\newblock {\em {A Wavelet Tour of Signal Processing}}.
\newblock Elsevier, 2009.

\bibitem{J.Medlock}
J.~Medlock and M.~Kot.
\newblock Spreading disease: integro-differential equations old and new.
\newblock {\em Mathematical Biosciences}, 184:201--222, 2003.

\bibitem{CCStolk2011}
C.~C. Stolk.
\newblock A preconditioner for the helmholtz equation based on adaptive phase
  space tiling.
\newblock Xrive, 2010.

\bibitem{Str}
J.~C. Strikwerda.
\newblock {\em {Finite Difference Schemes and Partial Differential Equations}}.
\newblock Wadsworth and Brooks, Cole Advanced Books and Software, Pacific
  Grove, California, 1989.

\bibitem{LNTre}
L.~N. Trefethen.
\newblock {\em {Spectral Methods in \mat}}.
\newblock SIAM, Philadelphia, 2000.

\bibitem{Cornelius.W.Oosterlee2001}
U.~Trottenberg, C.~W. Oosterlee, and A.~Schuller.
\newblock {\em Multigrid}.
\newblock Academic Press, 2001.

\bibitem{KUrban2009}
K.~Urban.
\newblock {\em {Wavelet Methods for Elliptic Partial Differential Equations}}.
\newblock Oxford University Press, 2009.

\bibitem{J.S.Walker}
J.~S. Walker.
\newblock {\em {Fast Fourier Transforms}}.
\newblock Second edition, CRC press, Boca Raton, New York, London, Tokyo, 1996.

\end{thebibliography}
\bibliographystyle{plain}
\end {document}